\documentclass[11pt,reqno]{amsart}

\usepackage{amsmath}
\usepackage{amsthm}
\usepackage{amssymb}
\usepackage{comment}

\usepackage{graphicx}

\theoremstyle{plain}                    

\theoremstyle{plain}
\newtheorem{theorem}{Theorem}[section]
\newtheorem{lemma}[theorem]{Lemma}
\newtheorem{proposition}[theorem]{Proposition}

\theoremstyle{definition}
\newtheorem{example}[theorem]{Example}
\newtheorem{definition}[theorem]{Definition}

\theoremstyle{remark}
\newtheorem{remark}[theorem]{Remark}

\DeclareMathOperator{\dist}{dist}
\DeclareMathOperator{\osc}{osc}
\DeclareMathOperator{\diam}{diam}

\begin{document}

\title[Infinity Laplace equation with Neumann boundary conditions]{Tug-of-war and infinity Laplace equation with vanishing Neumann boundary condition}
\author[T. Antunovi\'{c}]{Ton\'{c}i Antunovi\'{c}}
\address{Ton\'{c}i Antunovi\'{c}, Department of Mathematics,
University of California, Berkeley, CA 94720}
\email{tantun@math.berkeley.edu}

\author[Y. Peres]{Yuval Peres}
\address{Yuval Peres,
Microsoft Research,
Theory Group,
Redmond, WA 98052}
\email{peres@microsoft.com}

\author[S. Sheffield]{Scott Sheffield}
\address{Scott Sheffield, Department of Mathematics,
Massachusetts Institute of Technology,
Cambridge, MA 02139}
\email{sheffield@math.mit.edu}

\author[S. Somersille]{Stephanie Somersille}
\address{Stephanie Somersille, Department of Mathematics, University of Texas, Austin, TX 78712}
\email{steph@math.utexas.edu}


\maketitle

\begin{abstract}
We study a version of the stochastic ``tug-of-war'' game, played on graphs and smooth domains, with the empty set of terminal states. We prove that, when the running payoff function is shifted by an appropriate constant, the values of the game after $n$ steps converge in the continuous case and the case of finite graphs with loops. Using this we prove the existence of solutions to the infinity Laplace equation with vanishing Neumann boundary condition.
\end{abstract}

\section{Introduction}

\subsection{Background and previous results}
For a (possibly infinite) graph $G=(V,E)$, the stochastic tug-of-war game, as introduced in \cite{PSSW}, is a two-player zero-sum game defined as follows. At the beginning there is a token located at a vertex $x \in V$. At each step of the game players toss a fair coin and the winning player gets to move the token to an arbitrary neighbor of $x$. At the same time Player II pays Player I the value $f(x)$, where $f \colon V \to \mathbb{R}$ is a given function on the set of vertices, called the \textit{running payoff}. The game stops when the token reaches any vertex in a given set $W \subset V$, called the \textit{terminal set}. If $y \in W$ is the final position of the token, then Player II pays Player I a value of $g(y)$ for a given function $g \colon W \to \mathbb{R}$ called the \textit{terminal payoff}.
One can show that, when $g$ is bounded and either $f=0$, $\inf f>0$ or $\sup f<0$, this game has a value (Theorem 1.2 in \cite{PSSW}), which corresponds to the expected total amount that Player II pays to Player I when both players ``play optimally''.

Among other reasons, these games are interesting because of a connection between the game values and viscosity solutions of the infinity Laplace equation. Let $\Omega \subset \mathbb{R}^d$ be a domain (open, bounded and connected set) with $C^1$ boundary $\partial \Omega$, and let $f \colon \Omega \to \mathbb{R}$ and $g \colon \partial \Omega \to \mathbb{R}$ be continuous functions. Define the graph with the vertex set $\overline{\Omega}$ so that two points $x,y \in \overline{\Omega}$ are connected by an edge if and only if the intrinsic path distance between $x$ and $y$ in $\overline{\Omega}$ is less than $\epsilon$. Playing the game on this graph corresponds to moving the token from a position $x \in \Omega$ to anywhere inside the ball with the center in $x$ and radius $\epsilon$, defined with respect to the intrinsic path metric in $\overline{\Omega}$.
Consider this game with the running payoff $\epsilon^2f$, the terminal set $\partial \Omega$ and the terminal payoff $g$. By Dynamic programming principle if the value of this game exists then it is a solution to the finite difference equation 
\begin{equation*}
u(x) - \frac{1}{2}\left(\min_{B(x,\epsilon)}u + \max_{B(x,\epsilon)}u\right) = f(x),
\end{equation*}
for all $x \in \Omega$, and $u(y)=g(y)$,  for all $y \in \partial \Omega$.
In \cite{PSSW} it was shown that, under certain assumptions on the payoff function $f$, the game values  with step size $\epsilon$ converge as $\epsilon$ converges to zero appropriately. Moreover the limit $u$ is shown to be a viscosity solution to the non-linear partial differential equation 
\begin{equation*}
\left\{\begin{aligned}
& -\Delta_\infty u = f & \text{in} & \ \Omega,\\
& u = g & \text{on} & \ \partial \Omega.
\end{aligned}\right.
\end{equation*}
Using finite difference approach and avoiding probabilistic arguments, Armstrong and Smart  \cite{AS} obtained general existence results for this equation and the uniqueness for typical shifts of the function $f$. Several modifications of this game have also been studied, including using biased coins, which corresponds to adding a gradient term to the equation (see \cite{PPS}) and taking the terminal set $W$ to be a non-empty subset of $\partial \Omega$, which corresponds to Dirichlet boundary conditions on $W$ and vanishing Neumann boundary conditions on $\partial \Omega \setminus W$ (see \cite{ASS}, \cite{CGR_1} and \cite{CGR_2}). 

A crucial property of these games is the fact that the terminal set is non-empty which ensures that the game can stop in finite time. However to use the above connection in order to study the infinity Laplace equation with pure vanishing Neumann boundary conditions, one would have to consider this game without the terminal set. This is the content of this paper.

In the following two subsections we introduce the notation and give necessary definitions. In Section \ref{section: the discrete case} we study the tug-of-war games of finite horizons (defined below) on finite graphs. The results we obtain are used in Section \ref{section:the continuous case} to prove the existence of solutions of infinity Laplace equations with pure vanishing Neumann boundary conditions. Section \ref{section: uniqueness discussion} contains a discussion about uniqueness.

\subsection{Setting and notation}

All graphs that we consider in the text will be connected and of finite diameter (in graph metric), but we allow graphs to have an uncountable number of vertices (and vertices with uncountable degrees). 
However, for most of the main results we will need additional assumptions. 

\begin{definition}\label{def:length_graphs}
Let $(V,d)$ be a compact length space, that is, a compact metric space such that for all $x,y \in V$ the distance $d(x,y)$ is the infimum of the lengths of rectifiable paths between $x$ and $y$. For a fixed $\epsilon > 0$, the $\epsilon$-\textit{adjacency graph} is defined as a graph with the vertex set $V$, such that two vertices $x$ and $y$ are connected if and only if $d(x,y) \leq \epsilon$. When the value of $\epsilon$ is not important we will simply use the term adjacency graph.
\end{definition}

A particular example of $\epsilon$-adjacency graphs, already described in the previous subsection, corresponds to taking $(V,d)$ to be a closure of a domain $\Omega \subset \mathbb{R}^d$ with $C^1$ boundary $\partial \Omega$ and the intrinsic metric in $\overline{\Omega}$. This means that $d(x,y)$ is equal to the infimum of the lengths of rectifiable paths contained in $\overline{\Omega}$, between points  $x$ and $y$. We will call these graphs \textit{Euclidean} $\epsilon$-\textit{adjacency graphs}. Note that in the present work we will limit our attention to domains with $C^1$ boundary (meaning that for any $x \in \partial \Omega$ we can find open sets $U\subset \mathbb{R}^{n-1}$ and $V\subset \mathbb{R}^n$ containing $0$ and $x$ respectively, a $C^1$ function $\phi \colon U \to \mathbb{R}^{n-1}$ and an isometry of $\mathbb{R}^n$ which maps $(0,\phi(0))$ to $x$ and the graph of $\phi$ onto $V \cap \partial \Omega$).

While the fact that $(\overline{\Omega},d)$ is a metric space is fairly standard,
at this point we need to argue that this space is compact. Actually one can see that the topology induced by this metric space is the same as the Euclidean topology, and we only need to show that it is finer. To end this assume that $(x_n)$ is a sequence of points such that $\lim_n |x_n -x| = 0$, for some $x \in \overline{\Omega}$. If $x \in \Omega$ it is clear that $\lim_n d(x_n, x) = 0$. On the other hand if $x \in \partial \Omega$ then let $y_n$ be the closest point on $\partial \Omega$ to $x_n$. It is clear that $d(x_n,y_n) = |x_n-y_n|$ converges to zero as $n \to \infty$ and $\lim_n |x - y_n| = 0$. To complete the argument simply consider the paths between $x$ and $y_n$ contained in $\partial \Omega$, which are obtained by composing a $C^1$ parametrization of $\partial \Omega$ in a neighborhood of $x$ and an affine function. The lengths of these paths converge to zero.

Another interesting class of graphs will be finite graphs with loops at each vertex. Intuitively, these graphs might be thought of as $\epsilon$-adjacency graphs where $(V,d)$ is a finite metric space with integer valued lengths and $\epsilon = 1$.

Note that for two vertices $x$ and $y$ we write $x \sim y$ if $x$ and $y$ are connected with an edge. 
The graph metric between vertices $x$ and $y$ will be denoted by $\dist(x,y)$ in order to distinguish it from the metric in Definition \ref{def:length_graphs}. The diameter of a graph $G$ will be denoted by $\diam (G)$.
We consider the supremum norm on the space of functions on $V$, that is 
$\|u\| = \max_x |u(x)|$. When $V$ is compact length space, this norm makes $C(V,\mathbb{R})$, the space of continuous functions on $V$, a Banach  space. We also use the notation $B(x,\epsilon) = \{y \in V: d(x,y) \leq \epsilon\}$ (we want to emphasize that, in contrast with \cite{PSSW}, the balls $B(x,\epsilon)$ are defined to be closed) .

We consider a version of the tug-of-war game in which the terminal set is empty, but in which the game is stopped after $n$ steps.  
We say that this game has \emph{horizon} $n$. At each step, if the token is at the vertex $x$ Player II pays Player I value $f(x)$. If the final position of the token is a vertex $y$ then at the end Player II pays Player I value $g(y)$. Here $f$ and $g$ are real bounded functions on the set of vertices called the running and the terminal payoff.
Actually this game can be realized as the original stochastic tug-of-war game introduced in \cite{PSSW} played on the graph $G \times \{1,2,\dots ,n\}$ (the edges connecting vertices of the form $(v,i)$ and $(w,i+1)$, where $v$ and $w$ are neighbors in $G$), for the running payoff  $f(v,i) := f(v)$,  the terminal set  $V \times \{n\}$ and the terminal payoff  $g(v,n) := g(v)$.
Note that the same game with vanishing running payoff has appeared in a different context in \cite{CDP}. 

 Define strategy of a player to be a function that, for each $1 \leq k \leq n$, at the $k$-th step maps the previous $k$ positions and $k$ coin tosses to a vertex of the graph which neighbors the current position of the token. For a Player I strategy $\mathcal{S}_{\text{I}}$ and Player II strategy $\mathcal{S}_{\text{II}}$ define $F_n(\mathcal{S}_{\text{I}},\mathcal{S}_{\text{II}})$ as the expected payoff in the game of horizon $n$, when Players I and II play according to strategies $\mathcal{S}_{\text{I}}$ and $\mathcal{S}_{\text{II}}$ respectively.
Define the value for Player I as $u_{\text{I},n}=\sup_{\mathcal{S}_\text{I}}\inf_{\mathcal{S}_{\text{II}}}F_n(\mathcal{S}_{\text{I}},\mathcal{S}_{\text{II}})$, and the value for Player II as $u_{\text{II},n}=\inf_{\mathcal{S}_{\text{II}}}\sup_{\mathcal{S}_{\text{I}}}F_n(\mathcal{S}_{\text{I}},\mathcal{S}_{\text{II}})$. Note that we consider both $u_{\text{I},n}$ and $u_{\text{II},n}$ as functions of the initial position of the token. Intuitively $u_{\text{I},n}(x)$ is the supremum of the values that Player I can ensure to earn and $u_{\text{II},n}(x)$ is the infimum of the values Player II can ensure not to overpay, both in the game of horizon $n$ that starts from $x \in V$. 
It is clear that $u_{\text{I},0}=u_{\text{II},0}=g$  and one can easily check that  $u_{\text{I},n} \leq u_{\text{II},n}$.
In a game of horizon $n+1$ that starts at $x$, if Players I and II play according to the strategies $\mathcal{S}_{\text{I}}$ and $\mathcal{S}_{\text{II}}$ which, in the first step, push the token to $x_{\text{I}}$ and $x_{\text{II}}$ respectively, we have 
\begin{equation}\label{eq:game_value_step}
F_{n+1}(\mathcal{S}_{\text{I}},\mathcal{S}_{\text{II}})(x) = f(x) + \frac{1}{2}\Big(F_{2,n+1}(\mathcal{S}_{\text{I}},\mathcal{S}_{\text{II}})(x_{\text{I}}) + F_{2,n+1}(\mathcal{S}_{\text{I}},\mathcal{S}_{\text{II}})(x_{\text{II}})\Big),
\end{equation}
where $F_{2,n+1}(\mathcal{S}_{\text{I}},\mathcal{S}_{\text{II}})(y)$ is the expected payoff between steps $2$ and $n+1$ conditioned on having position $y$ in the second step of the game. It is easy to see that using the above notation $u_{\text{I},n}=\sup_{\mathcal{S}_\text{I}}\inf_{\mathcal{S}_{\text{II}}}F_{2,n+1}(\mathcal{S}_{\text{I}},\mathcal{S}_{\text{II}})$ and $u_{\text{II},n}=\inf_{\mathcal{S}_{\text{II}}}\sup_{\mathcal{S}_{\text{I}}}F_{2,n+1}(\mathcal{S}_{\text{I}},\mathcal{S}_{\text{II}})$, where  $\mathcal{S}_{\text{I}}$ and $\mathcal{S}_{\text{II}}$ are strategies for Player I and Player II in the game that lasts for $n+1$ steps. Now using induction in $n$ and \eqref{eq:game_value_step} one can check that $u_n:=u_{\text{I},n}=u_{\text{II},n}$ for any $n$, and that the sequence $(u_n)$ satisfies $u_0 = g$ and 
\begin{equation}\label{eq:recursion}
u_{n+1}(x) = \frac{1}{2}\Big(\min_{y \sim x}u_n(y) + \max_{y \sim x}u_n(y)\Big) + f(x).
\end{equation}
Furthermore the infima and the suprema in the definitions of $u_{\text{I},n}$ and $u_{\text{II},n}$ are  achieved for the strategies that at step $k$ of the game of horizon $n$ pull the  token to a neighbor that maximizes (minimizes) the value of $u_{n-k}$ (such a neighbor exists for finite degree graphs, and also in the case of an $\epsilon$-adjacency graph provided $u_{n-k}$ is known a priori to be continuous).

In this paper we will mainly study the described  game through the recursion \eqref{eq:recursion}.

\begin{remark}\label{rem:continuity}
In the case of $\epsilon$-adjacency graphs we will normally assume that the terminal and the running payoff are continuous functions on $V$ and heavily use the fact that the game values $u_n$ are continuous functions. To justify this it is enough to show that, if $u$ is a continuous function on $V$, then so are $\overline{u}^\epsilon(x)=\max_{B(x,\epsilon)}u$ and  $\underline{u}_\epsilon(x)=\min_{B(x,\epsilon)}u$. For this, one only needs to observe that for any two points $x,y \in V$ such that $d(x,y) < \delta$, any point in $B(x,\epsilon)$ is within distance of $\delta$ from some point in $B(y,\epsilon)$, and vice versa.
Now the (uniform) continuity of $\overline{u}^\epsilon$ and $\underline{u}_\epsilon$ follows from the uniform continuity of $u$, which holds by compactness of $V$. 
\end{remark}

For a game played on an arbitrary connected graph of finite diameter, if the sequence of game values $(u_n)$ converges pointwise, the limit $u$ is a solution to the equation
\begin{equation}\label{eq:equation_discrete}
u(x) - \frac{1}{2}(\min_{y \sim x}u(y) + \max_{y \sim x}u(y)) = f(x).
\end{equation}
The \textit{discrete infinity Laplacian} $\Delta_\infty u$ is defined at a vertex $x$ as the negative left hand side of the above equation.
As mentioned before, the case of Euclidean $\epsilon$-adjacency graphs is interesting because of the connection between the game values and the viscosity solutions of \eqref{eq:equation_continuous} defined in Definition \ref{def:viscosity_solutions}. To observe this it is necessary to scale the payoff function by the factor of $\epsilon^2$. Therefore in the case of $\epsilon$-adjacency graphs we define the $\epsilon$\emph{-discrete Laplace operator} as
\[
\Delta_\infty^\epsilon u(x) = \frac{(\min_{y \in B(x,\epsilon)}u(y) + \max_{y \in B(x,\epsilon)}u(y)) - 2u(x)}{\epsilon^2},
\]
and consider the equation
\begin{equation}\label{eq:equation_continuous_finite_step}
-\Delta_\infty^\epsilon u = f.
\end{equation}

\begin{remark}\label{rem:two_laplacians}
Observe that, compared to the discrete infinity Laplacian, we removed a factor $2$ from the denominator. This definition is more natural when considering the infinity Laplacian $\Delta_\infty$ described below. As a consequence we have that the pointwise limit $u$ of the game values $(u_n)_n$ played on an $\epsilon$-adjacency graph with the payoff function $\epsilon^2 f/2$ is a solution to \eqref{eq:equation_continuous_finite_step}.
\end{remark}

We will consider the infinity Laplace equation on a connected domain $\Omega$ with $C^1$ boundary $\partial \Omega$, with vanishing Neumann boundary conditions
\begin{equation}\label{eq:equation_continuous}
\left\{\begin{aligned}
& -\Delta_\infty u = f & \text{in} & \ \Omega,\\
&\nabla_\nu u = 0 & \text{on} & \ \partial \Omega.
\end{aligned}\right.
\end{equation}
Here $\nu$ (or more precisely $\nu(x)$) denotes the normal vector to $\partial \Omega$ at a point $x \in \partial \Omega$. 
The \emph{infinity Laplacian} $\Delta_\infty$ is formally defined as the second derivative in the gradient direction, that is
\begin{equation}\label{eq:infinity_laplacian}
\Delta_\infty u = |\nabla u|^{-2} \sum_{i,j}u_{x_i}u_{x_ix_j}u_{x_j}.
\end{equation}
We will define the solutions of \eqref{eq:equation_continuous} and prove the existence in the following viscosity sense.
First define the operators $\Delta_\infty^+$ and $\Delta_\infty^-$ as follows. For a twice differentiable function $u$ and a point $x$ such that $\nabla u(x) \neq 0$ define $\Delta_\infty^+u$ and $\Delta_\infty^-u$
to be given by \eqref{eq:infinity_laplacian}, that is $\Delta_\infty^+u(x) = \Delta_\infty^-u(x) = \Delta_\infty u(x)$. For $x$ such that $\nabla u(x) = 0$ define $\Delta_\infty^+u(x) = \max\{\sum_{i,j}u_{x_ix_j}(x)\mathbf{v}_i\mathbf{v}_j\}$ and $\Delta_\infty^-u(x) = \min\{\sum_{i,j}u_{x_ix_j}(x)\mathbf{v}_i\mathbf{v}_j\}$, where the maximum and the minimum are taken over all vectors $\mathbf{v}=(\mathbf{v}_1, \dots ,\mathbf{v}_d)$ of Euclidean norm $1$. 

\begin{definition}\label{def:viscosity_solutions}
A continuous function $u\colon \overline{\Omega}\to \mathbb{R}$ is said to be a \textit{(viscosity) subsolution} to \eqref{eq:equation_continuous} if for any function $\varphi \in C^\infty(\overline{\Omega})$ (infinitely differentiable function on an open set containing $\overline{\Omega}$) and a point $x_0 \in \overline{\Omega}$, such that $u-\varphi$ has a strict local maximum at $x_0$, we have either 
\begin{itemize}
\item[a)] $-\Delta_\infty^+\varphi(x_0) \leq f(x_0)$ or 
\item[b)] $x_0 \in \partial \Omega$ and $\nabla_{\nu(x_0)} \varphi(x_0) \leq 0$. 
\end{itemize}
A continuous function $u\colon \overline{\Omega} \to \mathbb{R}$ is said to be a \textit{(viscosity) supersolution} if $-u$ is a (viscosity) subsolution when $f$ is replaced by $-f$ in \eqref{eq:equation_continuous}. A continuous function $u$ is said to be a \textit{(viscosity) solution} to \eqref{eq:equation_continuous} if it is both a subsolution and a supersolution.
\end{definition}

Note  that the notion of (sub, super)solutions does not change if one replaces the condition $\varphi \in C^\infty(\overline{\Omega})$ with $\varphi \in C^2(\overline{\Omega})$.

\begin{remark}\label{rem:strong_viscosity_solutions}
The above definition, while having the advantage of being closed under taking limits of sequences of solutions, might be slightly unnatural because the condition in a) is sufficient for $x_0 \in \partial \Omega$ at which $u-\varphi$ has a strict local maximum. Following \cite{CHI} we can define the \textit{strong (viscosity) subsolution} as a continuous function $u$ such that, for any $\varphi \in C^\infty(\overline{\Omega})$ and any $x_0 \in \overline{\Omega}$, at which $u -\varphi$ has a strict local maximum,  we have 
\begin{itemize}
\item[a')] $-\Delta_\infty^+\varphi(x_0) \leq f(x_0)$, if $x_0 \in \Omega$, 
\item[b')] $\nabla_\nu \varphi(x_0) \leq 0$, if $x_0 \in \partial \Omega$. 
\end{itemize}
Strong (viscosity) supersolutions and solutions are defined analogously. While it is clear that the requirements in this definition are stronger than those in Definition \ref{def:viscosity_solutions}, it can be shown that, when $\Omega$ is a convex domain, any (sub, super)solution is also a strong (sub, super)solution. To show this assume that $u$ is a viscosity subsolution to \eqref{eq:infinity_laplacian} in the sense of Definition \ref{def:viscosity_solutions} and let $x \in \partial\Omega$ and $\varphi \in C^\infty(\overline{\Omega})$ be such that $u-\varphi$ has a strict local maximum at $x$ and $\nabla_\nu \varphi(x) > 0$. Without the loss of generality we can assume that $x=0$ and that the normal vector to $\partial \Omega$ at $0$ is  $\nu=-\mathbf{e}_d$, where $\mathbf{e}_d$ is the $d$-th coordinate vector. Thus by the convexity, the domain $\Omega$ lies above the coordinate plane $x_d=0$. Now define the function $\phi \in C^\infty(\overline{\Omega})$ as $\phi(y) =\varphi(y) + \alpha y_d - \beta (y_d)^2$, for positive $\alpha$ and $\beta$. Since $\nabla \phi (0) = \nabla \varphi (0) + \alpha \mathbf{e}_d$, for $\alpha$ small enough we still have $\nabla_\nu \phi(0) > 0$. Moreover the Hessian matrix of $\phi$ is the same as that of $\varphi$, with the exception of the $(d,d)$-entry which is decreased by $2\beta$. Since $\nabla_\nu \phi(0) > 0$ and $\nabla_\nu \phi(0)$ does not depend on $\beta$, for $\beta$ large enough we will have $-\Delta_\infty^+\phi(0) > f(0)$. Moreover since we can find an open set $U$ such that $\varphi(y) \leq \phi(y)$ for all $y \in \overline{\Omega}\cap U$, we have that $u-\phi$ has again a strict local maximum at $0$.
 Since it doesn't satisfy the conditions in Definition \ref{def:viscosity_solutions}, this leads to a contradiction.
\end{remark}

\subsection{Statements of results}

We want to study the values of games as their horizons tend to infinity. Clearly taking payoff function $f$ to be of constant sign will make the game values diverge. Since increasing the payoff function by a constant $c$ results in the increase of the  value of the game  of horizon $n$ by $nc$, 
 the most we can expect is that we can find a (necessarily unique) shift $f+c$ of the payoff function $f$ for which the game values behave nicely.
The first result in this direction is the following theorem which holds for all connected graphs of finite diameter.

\begin{theorem}\label{thm:first_result}
For any connected graph $G=(V,E)$ of finite diameter and any bounded function $f \colon V \to \mathbb{R}$ there is a constant $c_f$, such that the following holds: 
For any bounded function $u_0 \colon V \to \mathbb{R}$, if $(u_n)$ is the sequence of game values with the terminal and running payoffs $u_0$ and $f$ respectively, then the sequence of functions $(u_n-nc_f)$ is bounded. 
\end{theorem}

We will call $c_f$ from Theorem \ref{thm:first_result} the \textit{Player I's long term advantage} for function $f$. 
For both adjacency graphs and finite graphs with loops we have convergence of the game values.

\begin{theorem}\label{thm:graphs_with_loops}
Let $G=(V,E)$ be either an adjacency graph, or a finite graph with a loop at each vertex. Let $f, u_0 \colon V \to \mathbb{R}$ be functions on the set of vertices, which are assumed to be continuous if $G$ is an adjacency graph. Assume $c_f =0$.
In a game played on $G$ with the terminal and running payoffs $u_0$ and $f$ respectively, the sequence of game values $(u_n)$ is uniformly convergent.
\end{theorem}

The following theorem gives the correspondence between the tug-of-war games and the equation \eqref{eq:equation_discrete}.
 While for adjacency graphs and finite graphs with loops this is a straightforward corollary of Theorem \ref{thm:graphs_with_loops}, this result also holds for all finite graphs, even when Theorem \ref{thm:graphs_with_loops} may fail to hold (see Example \ref{example:no_convergence}).

\begin{theorem}\label{thm:solutions_on_general_graphs}
Let $G=(V,E)$ be either an adjacency graph, or a finite graph. Let $f, u_0 \colon V \to \mathbb{R}$ be functions on the set of vertices, which are assumed to be continuous if $G$ is an adjacency graph.
Then the equation \eqref{eq:equation_discrete} has a solution $u$ if and only if $c_f=0$.
\end{theorem}

Let $V$ be a compact length space and a consider a function $f \in C(V)$. To emphasize the dependence on $\epsilon$ define $c_f(\epsilon)$ as the Player I's long term advantage for a game played on the $\epsilon$-adjacency graph defined on $V$ with the running payoff $f$.
As already mentioned, to study the limiting case $\epsilon \downarrow 0$ for $\epsilon$-adjacency graphs, we need to scale the running payoff function by a factor of $\epsilon^2$, that is, the we take $\epsilon^2 f$ as the running payoff function. Note that the Player I's long term advantage corresponding to this game is equal to $\epsilon^2 c_f(\epsilon)$. 



The first problem one encounters is the fact that $c_f(\epsilon)$ depends on the value of $\epsilon$ (see Example \ref{example:shifting_parameters_continuous}). The following theorem gives meaning to the notion of Player I's long term advantage in the continuous case.

\begin{theorem}\label{thm:convergence_of_shifting_constants}
For any compact length space $V$ and any continuous function $f \colon V \to \mathbb{R}$ the limit $\lim_{\epsilon \downarrow 0}c_f(\epsilon)$ exists.
\end{theorem}

We will denote the limit from the above theorem by $\overline{c}_f=\lim_{\epsilon \downarrow 0}c_f(\epsilon)$.

\begin{theorem}\label{thm:general_continuous_solutions}
Let $V$ be a compact length space, and $(\epsilon_n)$ a sequence of positive real numbers converging to zero. Any sequence $(u_n)$  of continuous functions on $V$ satisfying $-\Delta_\infty^{\epsilon_n}u_n = f- c_f(\epsilon_n)$ and such that $0$ is in the range of $u_n$ for all $n$, has a subsequence converging to a Lipshitz continuous function. Moreover the Lipshitz constant is bounded by a universal constant multiple of $\diam (V) \|f\|$.
\end{theorem}

For Euclidean $\epsilon$-adjacency graphs, the limits from Theorem \ref{thm:general_continuous_solutions}  give us viscosity solutions of \eqref{eq:equation_continuous}.

\begin{theorem}\label{thm:existence_of_continuous_solutions}
Let  $\Omega$ be a domain of finite diameter with $C^1$ boundary $\partial \Omega$ and $f \colon \overline{\Omega} \to \mathbb{R}$ a continuous function, such that $\overline{c}_f=0$. Then the equation \eqref{eq:equation_continuous} has a viscosity solution $u$ which is Lipshitz continuous, with Lipshitz constant depending on $\Omega$ and the norm $\|f\|$.
\end{theorem}

It is natural to expect the existence of viscosity solutions to \eqref{eq:equation_continuous} only for one shift of the function $f$. This is proven in the following theorem for convex domains $\Omega$.

\begin{theorem}\label{thm:uniqueness_of_shifting_constants}
Let  $\Omega$ be a convex domain of finite diameter with  $C^1$ boundary $\partial \Omega$ and $f \colon \overline{\Omega} \to \mathbb{R}$ a continuous function. Then the equation \eqref{eq:equation_continuous} has a viscosity solution $u$ if and only if $\overline{c}_f=0$.
\end{theorem}

\begin{remark}\label{rem:shifting constant scaling}
Directly from Theorem \ref{thm:first_result} one can deduce that $c_{\lambda f}= \lambda c_f$ and $c_{f+\lambda} = c_f +\lambda$ for any $\lambda \in \mathbb{R}$. For compact length spaces, after taking an appropriate limit, we obtain the same properties for $\overline{c}_f$. Thus Theorem \ref{thm:solutions_on_general_graphs} tells us that, under its  assumptions, for any function $g$ on the vertex set (continuous in the case of adjacency graphs), there is a unique constant $c$, such that 
 equation \eqref{eq:equation_discrete} can be solved when $f =g -c$. Theorems \ref{thm:existence_of_continuous_solutions} and \ref{thm:uniqueness_of_shifting_constants} tell us that any function $g \in C(\overline{\Omega})$ can be shifted to obtain a function $f \in C(\overline{\Omega})$ for which \eqref{eq:equation_continuous} can be solved, and that this shift  is unique when $\Omega$ is convex.
\end{remark}

\begin{remark}\label{eqm:uniqueness}
In the case of finite graphs and for a fixed $f$ (such that $c_f =0$), the solutions to the equation \eqref{eq:equation_discrete} are not necessarily unique (even in the case of finite graphs with self loops). A counterexample and a discussion about the continuous case is given in Section \ref{section: uniqueness discussion}.
\end{remark}

\section{The discrete case}\label{section: the discrete case}

Since the terminal payoff can be understood as the  value of the game of horizon $0$, we will not explicitly mention the terminal payoff when it is clear from the context.

\begin{lemma}\label{lemma:comparison}
Let $G=(V,E)$ be a connected graph of finite diameter and let $f$, $g$, $u_0$ and $v_0$ be bounded functions on $V$.
\begin{itemize}
\item[(i)] Let $(u_n)$ and $(v_n)$ be the sequences of values of games played with the running payoff $f$. If $u_0 \leq v_0$, then $u_n \leq v_n$ for all $n > 0$. Furthermore, if for some $c \in \mathbb{R}$ we have $v_0=u_0+c$, then $v_n=u_n+c$ for all $n>0$.
\item[(ii)] Let $(u_n^1)$ and $(u_n^2)$ be sequences of values of games played with the terminal payoffs $u_0^1=u_0^2=u_0$ and the running payoffs $f$ and $g$ respectively. If $f \leq g$ then $u_n^1 \leq u_n^2$, for all $n$. Furthermore if for some $c \in \mathbb{R}$ we have $g=f+c$, then $u_n^2 = u_n^1 + nc$ for all $n>0$.
\end{itemize}
\end{lemma}

\begin{proof}
All statements are easy to verify by induction on $n$ using relation \eqref{eq:recursion}.
\end{proof}

\begin{lemma}\label{lemma:boundedness_of_differences}
For a connected graph $G=(V,E)$ of finite diameter and bounded functions $f,u_0 \colon V \to \mathbb{R}$, let $(u_n)_n$ be the sequence of values of games played on $G$  with running payoff $f$. 
Then for all $n \geq 0$ we have
\[
\max u_n - \min u_n \leq (\max u_0 - \min u_0) + \diam(G)^2 (\max f - \min f). 
\]
\end{lemma}

\begin{proof}
Consider the sequence of game values $(v_n)$  played with the running payoff $f$ and zero terminal payoff. From part (i) of Lemma \ref{lemma:comparison} we get
$v_n + \min u_0 \leq u_n \leq v_n + \max u_0$.
This implies that 
\[
\max u_n - \min u_n \leq (\max v_n - \min v_n) + (\max u_0 - \min u_0).
\] 
From this it's clear that it is enough to prove the claim when $u_0=0$. 
Furthermore, by part (ii) of Lemma \ref{lemma:comparison} it is enough to prove the claim for an arbitrary shift of the payoff function $f$, and therefore we assume that $\min f = 0$.
This implies that
$u_n \geq 0$, for  $n \geq 0$. Now, for a fixed $n$, by Lemma \ref{lemma:comparison} (i), playing the game of horizon $n-k$ with the running payoff $f$ and the terminal payoffs $u_k$ and $0$, gives the game values $u_{n}$ and $u_{n-k}$ respectively, and
\begin{equation}\label{eq: weak tightness on fixed step 1}
u_{n-k} \leq  u_n.
                            \end{equation}

Fix a vertex $z \in V$.
For a vertex $y\in V$ pick a neighbor $z(y) \in V$ of $y$ so that $\dist(z(y),z)= \dist(y,z)-1$ (if $y$ and $z$ are neighbors then clearly $z(y)=z$).
Let $\mathcal{S}_{\text{II},k}^0$ be the optimal strategy for Player II in the game of horizon $k$. For any Player I strategy $\mathcal{S}_{\text{I},k}$  for a game of horizon $k$, we have  $F_k(\mathcal{S}_{\text{I},k},\mathcal{S}_{\text{II},k}^0) \leq u_k$.
Now define the ``pull towards $z$'' strategy $\mathcal{S}_{\text{II}}$ for a game of length $n$ as follows. At any step of the game if the token is at the vertex $y \neq z$ and if $z$ is not among the past positions of the token, then strategy $\mathcal{S}_{\text{II}}$ takes the token to the vertex $z(y)$. If $T$ is the first time at which the token is at the vertex $z$, at this point Player II starts playing using the strategy $\mathcal{S}_{\text{II},n-T}^0$.
If $X_t$ is the position of the token at time $t$, then it can be easily checked that for $Y_t=(\diam (G) - \dist(X_t,z))^2 -t$, the process $Y_{t\wedge T}$ is a submartingale, with uniformly bounded differences. Moreover the stopping time $T$ has a finite expectation since it is bounded from above by the first time that Player II has won $\diam(G)$ consecutive coin tosses
(partition coin tosses into consecutive blocks of length $\diam (G)$ and notice that the order of the first block in which Player II wins all the coin tosses has exponential distribution with mean  $2^{\diam (G)}$). Therefore applying the optional stopping theorem we get $\mathbb{E}(Y_T) \geq \mathbb{E}(Y_0)$,  hence $\mathbb{E}(T) \leq \diam(G)^2$.
Now consider the game in which Player I plays optimally and Player II plays according to the above defined strategy $\mathcal{S}_{\text{II}}$.
Since each move in the optimal strategies depends only on the current position of the token, by the independence of the coin tosses, we have that conditioned on $T=k$, the expected payoff in steps $k+1$ to $n$ is bounded from above by $u_{n-k}(z)$.
Clearly the total  payoff in the first $k$ steps is bounded from above by $k \max f$. For $x \neq z$ the strategy $\mathcal{S}_{\text{II}}^0$ is suboptimal and 
\[
u_n(x) \leq \sum_{k=1}^n \mathbb{P}(T=k)(k \max f + u_{n-k}(z)) + \mathbb{P}(T\geq n)n \max f.
\] 
Since $f$ is a non-negative function, so is $u_{n}$ for any $n$.
Using this with \eqref{eq: weak tightness on fixed step 1} we get
\begin{multline*}
u_n(x) \leq \sum_{k=1}^n \mathbb{P}(T=k) k\max f + u_n(z) +  \mathbb{P}(T\geq n)n \max f \\ \leq u_n(z) + \mathbb{E}(T)\max f.
\end{multline*}
Since $x$ and $z$ are arbitrary and $\mathbb{E}(T) \leq \diam (G)^2$ for all $x$  and $z$, the claim follows.
\end{proof}

\begin{proof}[Proof of Theorem \ref{thm:first_result}]
As in the proof of Lemma \ref{lemma:boundedness_of_differences}, we can assume that $u_0=0$.
Denote $M_k=\max u_k$ and $m_k=\min u_k$. By part (i) of Lemma \ref{lemma:comparison}, playing the game with the constant terminal payoff $M_k$ gives the sequence of game values $(u_n+M_k)_n$. 
Comparing this to the game with the terminal payoff $u_k$ we obtain $u_{n+k}(x) \leq u_n(x) + M_k$.
Taking maximum over all vertices $x$ leads to the subaditivity of the sequence $(M_n)$, that is $M_{n+k} \leq M_n + M_k$. In the same way we can prove that the sequence $(m_n)$ is superaditive. By Lemma \ref{lemma:boundedness_of_differences} we can find a constant $C$ so that $M_n-m_n \leq C$ for any $n$, and thus we can define
\begin{equation}\label{eq:definition_of_shifting_constant_discrete_case}
c_f := \lim_n\frac{M_n}{n} = \inf_n \frac{M_n}{n} = \lim_n \frac{m_n}{n} = \sup_n \frac{m_n}{n}.
\end{equation}
Then, for any $n \geq 0$ we have
\[
nc_f \leq M_n \leq m_n + C \leq nc_f + C,
\]
and therefore, for any $x \in V$
\[
|u_n(x) -nc_f| \leq \max\{|M_n-nc_f|, |m_n-nc_f|\} \leq C.
\]
\end{proof}

For an arbitrary graph $G=(V,E)$ and a function $f$ on $V$, define the (non-linear) operator $A_f$ acting on the space of functions on $V$, so that for each $x \in V$
\[
A_f u(x) = \frac{1}{2}\Big(\max_{y \sim x} u(y) + \min_{y \sim x} u(y) \Big) + f(x).
\]

\begin{lemma}\label{lemma:lipschitz_properties}
Assume $G=(V,E)$ is either 
 an adjacency graph or a finite graph. Let $f$, $u$ and $v$ be functions on $V$, which are also assumed to be continuous if $G$ is an adjacency graph.
Then we have
\begin{equation}
  \label{eq:lipschitz_properties_1}
\min(v-u) \leq \min(A_fv-A_fu) \leq \max(A_fv-A_fu) \leq \max(v-u),
\end{equation}
and 
\begin{equation}
  \label{eq:eq:lipshitz_properties_2}
\|A_fv-A_fu\| \leq \|v-u\|.
\end{equation}
Moreover for $x \in V$ we have $A_fv(x) - A_fu(x) = \max(v-u)$ if and only if for any two neighbors $y_1$ and $y_2$ of $x$ such that $u(y_1)=\min_{y \sim x}u(y)$,  and $v(y_2)=\max_{y \sim x}v(y)$, we also have $v(y_1)=\min_{y \sim x}v(y)$,  and $u(y_2)=\max_{y \sim x}u(y)$
 and $v(y_i) - u(y_i) = \max(v-u)$, for $i \in \{1,2\}$. 
\end{lemma}

\begin{proof}
Fix a vertex $x \in V$ and note that $\max_{y \sim x}v(y) \leq \max_{y \sim x}u(y) + \max(v-u)$ and $\min_{y \sim x}v(y) \leq \min_{y \sim x}u(y) + \max(v-u)$. Adding these inequalities one  obtains $A_f v(x) \leq A_fu(x) + \max(v-u)$. The inequality $\min(v-u) \leq \min(A_fv-A_fu)$ now follows by replacing $u$ and $v$ by $-u$ and $-v$ respectively, and \eqref{eq:eq:lipshitz_properties_2} follows directly from \eqref{eq:lipschitz_properties_1}. It is clear that the equality $A_fv(x)-A_fu(x) = \max(v-u)$ holds if and only if both 
\begin{equation}\label{eq: equality in lipschitz 1}
\max_{y \sim x}v(y) = \max_{y \sim x}u(y) + \max(v-u)
\end{equation} 
and 
\begin{equation}\label{eq: equality in lipschitz 2}
\min_{y \sim x}v(y) = \min_{y \sim x}u(y) + \max(v-u)
\end{equation} 
hold. It is obvious that the conditions in the statement are sufficient for \eqref{eq: equality in lipschitz 1} and \eqref{eq: equality in lipschitz 2} to hold, and it is only left to be proven that these conditions are also necessary. To end this assume that both \eqref{eq: equality in lipschitz 1} and \eqref{eq: equality in lipschitz 2} hold and take $y_1$ and $y_2$ to be arbitrary neighbors of $x$ such that $u(y_1) = \min_{y \sim x}u(y)$ and $v(y_2) = \max_{y \sim x}v(y)$. Clearly we have
\[
\min_{y \sim x}v(y)  \leq v(y_1) \leq u(y_1) + \max(v-u) = \min_{y \sim x}u(y)+\max(v-u),
\]
and moreover all the inequalities in the above expression must be equalities. This implies both $v(y_1) = u(y_1) + \max(v-u)$ and $v(y_1) = \min_{y \sim x}v(y)$. The claim for $y_2$ can be checked similarly.
\end{proof}

The following proposition proves Theorem \ref{thm:graphs_with_loops} in the case of finite graphs with loops.
For a sequence of game values $(u_n)_n$ with the running payoff $f$, define $M_n^f(u_0) = \max_{x \in V}(u_n(x)-u_{n-1}(x))$ and $m_n^f(u_0) = \min_{x \in V}(u_n(x)-u_{n-1}(x))$.  

\begin{proposition}\label{lemma:subsequential_to_complete_convergence}
Under the assumptions of Theorem \ref{thm:graphs_with_loops}
the sequence of game values converges if it has a convergent subsequence.
\end{proposition}

\begin{proof}
If $M_n^f(u_0)= -\delta < 0$ for some $n \geq 1$, we have $u_n \leq u_{n-1} - \delta$ and by applying part (i) of Lemma \ref{lemma:comparison} we obtain $u_m \leq u_{m-1} - \delta$, for any $m \geq n$. This is a contradiction with the assumption that $c_f=0$. 
Therefore $M_n^f(u_0) \geq 0$ and similarly $m_n^f(u_0) \leq 0$. By Lemma \ref{lemma:lipschitz_properties} the sequences $(M_n^f(u_0))_n$ and $(m_n^f(u_0))_n$ are bounded  and non-increasing and non-decreasing respectively and therefore they converge. 

Let $w$ be the limit of a subsequence of $(u_n)_n$. Assume for the moment that $M_1^f(w)=m_1^f(w)=0$, or equivalently $A_fw=w$. Now Lemma \ref{lemma:lipschitz_properties}   implies
\[
\|u_{n+1}-w\| = \|A_fu_n - A_fw\| \leq \|u_n-w\|.
\]
Therefore $\|u_{n}-w\|$ is decreasing in $n$ and, together with the fact that $0$ is its accumulation point, this yields $\lim_n \|u_{n}-w\| = 0$. Therefore it is enough to prove that $M_1^f(w)=m_1^f(w)=0$. The rest of the proof will be dedicated to showing $M_1^f(w)=0$ (the claim $m_1^f(w)=0$ following analogously).

First we prove that $(M_n^f(w))$ is a constant sequence. Assume this is not the case, that is for some $k$ we have $M_{k+1}^f(w) < M_k^f(w)$. For any $n \geq 1$ the mapping $v \mapsto M_n^f(v)$ is continuous and
therefore we can find a neighborhood $\mathcal{U}$ of $w$ ($\mathcal{U} \subset C(V,\mathbb{R})$ in the adjacency case)
 and $\delta>0$ such that $M_{k+1}^f(v) < M_k^f(v) - \delta$, for any $v \in \mathcal{U}$. Observing that $M_k^f(u_n) = M_{n+k}^f(u_0)$, and that $u_n \in \mathcal{U}$ for infinitely many positive integers $n$, we have  $M_{\ell+1}^f(u_0) < M_\ell^f(u_0) - \delta$ for infinitely many positive integers $\ell$. This is a contradiction with the fact that $(M_n^f(u_0))_n$ is a nonnegative decreasing sequence. 

Now let $M=M_n^f(w) \geq 0$ and denote by $(w_n)$ the sequence of game values with terminal and running payoffs $w_0=w$ and $f$ respectively. Define the compact sets $V_n=\{x \in V: w_{n+1}(x) = w_n(x)+M \}$ and $t_n = \min_{x \in V_n}w_n(x)$. Taking $x \in V_n$ we have $M=w_{n+1}(x) - w_n(x) = \max(w_n-w_{n-1})$ and thus we can apply Lemma \ref{lemma:lipschitz_properties} to find $y \sim x$ such that $w_n(y) = \min_{z \sim x}w_n(z)$ and $y \in V_{n-1}$. Because the graph $G$ satisfies $x \sim x$ for any vertex $x$, we obtain 
\begin{equation}\label{eq:use_of_loops}
w_n(x) \geq w_n(y) = w_{n-1}(y)+M \geq t_{n-1}+M,
\end{equation}
for any $x \in V_n$.
Taking the minimum over $x \in V_n$ yields $t_{n} \geq t_{n-1}+M$. For $M > 0$ this is a contradiction with the boundedness of the sequence $(w_n)$, which in turn follows from $c_f=0$. 
Thus $M=0$ which proves the statement.
\end{proof}

\begin{remark}\label{remark:use_of_loops}
Note that in the case of finite graphs, the first inequality in \eqref{eq:use_of_loops} is
 the only place where loops were used.
\end{remark}

The existence of accumulation points will follow from Lemma \ref{prop:equicontinuity}, which in turn will use the following lemma. Note that these two lemmas can replace the last paragraph in the proof of Lemma \ref{lemma:subsequential_to_complete_convergence}. However we will leave the proof of Lemma \ref{lemma:subsequential_to_complete_convergence} as it is, since it gives a shorter proof of Theorem \ref{thm:graphs_with_loops} for finite graphs with loops.

\begin{lemma}\label{lemma:convergence_of_differences}
Under the assumption of Theorem \ref{thm:graphs_with_loops} we have
\[
\lim_n (u_{n+1}-u_n)  = 0.
\]
\end{lemma}

\begin{proof}
We will prove that $\lim_n \max(u_{n+1}-u_n) = 0$. The claim then follows from the fact that $\lim_n \min(u_{n+1}-u_n) = 0$ which follows by replacing $u_n$ by $-u_n$ and $f$ by $-f$.

First assume that for some real numbers $\lambda_1$ and $\lambda_2$, a vertex $x \in V$ and a positive integer $n$ we have $u_{n+1}(x) -u_n(x) \geq \lambda_1$ and $M_{n}^f(u_0) \leq \lambda_2$. Since $\max_{z \sim x}u_n(z) - \max_{z \sim x}u_{n-1}(z) \leq \lambda_2$, by \eqref{eq:recursion} we see that 
\[
\min_{z \sim x}u_{n}(z)-\min_{z \sim x}u_{n-1}(z) \geq 2\lambda_1 - \lambda_2.
\] 
This implies that for a vertex $y \sim x$, such that $u_{n-1}(y) = \min_{z \sim x}u_{n-1}(z)$, we have 
\begin{equation}\label{eq:convergence_of_differences_step}
\min\{u_n(x),u_n(y)\} \geq  \min_{z \sim x}u_{n}(z) \geq u_{n-1}(y) + 2\lambda_1 - \lambda_2.
\end{equation}
We will inductively apply this simple argument to prove the statement. 

As argued in the proof of Proposition \ref{lemma:subsequential_to_complete_convergence} the sequence $(M_n^f(u_0))$ is non-increasing and nonnegative and therefore converges to $M=\lim_nM_n^f(u_0) \geq 0$. For a fixed $\delta>0$ let $n_0$ be an integer such that $M_n^f(u_0) \leq M + \delta$, for all $n \geq n_0$. For a given positive integer $k$, let $x_0$ be a point such that $u_{n_0+k}(x_0)-u_{n_0+k-1}(x_0) \geq M$. Then applying the reasoning that leads  to \eqref{eq:convergence_of_differences_step} for $\lambda_1=M$ and $\lambda_2=M+\delta$, we can find a point $x_1$ such that 
\[
\min\{u_{n_0+k-1}(x_0),u_{n_0+k-1}(x_1)\} \geq u_{n_0+k-2}(x_1)+(M-\delta).
\]
If $k \geq 3$ we can apply the same argument for functions $u_{n_0+k-1}$, $u_{n_0+k-2}$ and $u_{n_0+k-3}$, point $x_1$, $\lambda_1=M-\delta$ and $\lambda_2=M+\delta$.
Inductively repeating this reasoning we obtain a sequence of points $(x_\ell)$, $1 \leq \ell \leq k-1$ such that 
\[
\min\{u_{n_0+k-\ell}(x_{\ell-1}), u_{n_0+k-\ell}(x_\ell)\} \geq u_{n_0+k-\ell-1}(x_\ell) + M-(2^{\ell}-1)\delta.
\]
Summing the inequalities 
\[
u_{n_0+k-\ell}(x_{\ell-1}) \geq u_{n_0+k-\ell-1}(x_\ell) + M-(2^{\ell}-1)\delta,
\]
for $1 \leq \ell \leq k-1$ and $u_{n_0+k}(x_0)-u_{n_0+k-1}(x_0) \geq M$ leads to
\begin{equation}\label{eq:going_back_1}
u_{n_0+k}(x_0) \geq u_{n_0}(x_{k-1}) + kM - 2^k\delta,
\end{equation}
for all $k \geq 1$. Taking $k(\delta)$ to be the smallest integer larger than $\frac{\log (M/\delta)}{\log 2}$ we obtain 
\begin{equation}\label{eq:going_back_2}
u_{n_0+k(\delta)}(x_0) \geq u_{n_0}(x_{k(\delta)-1}) + M\Big(\frac{\log(M/\delta)}{\log 2}-2\Big).
\end{equation}
If $M>0$ then $\lim_{\delta \downarrow 0}M\Big(\frac{\log(M/\delta)}{\log 2}-2\Big)=\infty$, which by \eqref{eq:going_back_2} implies that the sequence $(u_n)$ is unbounded. This is a contradiction with the assumption that $c_f=0$.
\end{proof}

\begin{lemma}\label{prop:equicontinuity}
The sequence of game values $(u_n)$ for a game on an adjacency graph is an equicontinuous sequence of functions.
\end{lemma}

The proof of this lemma uses an idea similar to the proof of Lemma \ref{lemma:convergence_of_differences}. We obtain a contradiction by constructing a sequence of points along which the function values will be unbounded. Since the induction step is more complicated we put it into a separate lemma. First define the oscillation of a continuous function $v \colon V \to \mathbb{R}$ as $\osc(v,\delta) = \sup_{d(x,y) \leq \delta}|v(x)-v(y)|$.

\begin{lemma}\label{lemma:equicontinuity_help}
Let $(u_n)$ be a sequence of game values played on an adjacency graph.
Assume that for positive real numbers $\lambda_1$, $\lambda_2$, $\lambda_3$, $\rho < \epsilon$ and a positive integer $n$ we have
\begin{equation}\label{eq:equicontinuity_assumptions}
\osc(f,\rho) \leq \lambda_1, \ \osc(u_n,\rho) \leq \lambda_2, \ \text{and} \ \   u_n  \leq u_{n+1} + \lambda_3.
\end{equation}
Let $(x,y)$ be a pair of points which satisfies  $d(x,y) < \rho$ and $u_{n+1}(x)-u_{n+1}(y) \geq \delta$, for some $\delta> 0$. Then there are points $(x_1,y_1)$ which satisfy 
 $d(x_1,y_1) < \rho$, and the inequalities
\begin{equation}
  \label{eq:equicontinuity_new_pair_1}
u_n(x_1) - u_n(y_1) \geq  2\delta-2\lambda_1 - \lambda_2,
\end{equation}
 and 
\begin{equation}
  \label{eq:equicontinuity_new_pair_2}
u_{n+1}(y)-  u_n(y_1) \geq 2\delta-2\lambda_1 - \lambda_2 - \lambda_3.
\end{equation}
\end{lemma}

\begin{proof}
If $\delta \leq \lambda_1 + \lambda_2/2$ then consider the set $S=\{z: u_n(z) \leq u_{n+1}(y)+\lambda_3\}$, which is nonempty by the last condition in \eqref{eq:equicontinuity_assumptions}. If $S$ is equal to the whole space $V$, then \eqref{eq:equicontinuity_new_pair_2} will be satisfied automatically, and take $x_1$ and $y_1$ to be any points such that $d(x_1,y_1) < \rho$ and  $u_n(x_1) \geq  u_n(y_1)$ (so that \eqref{eq:equicontinuity_new_pair_1} is satisfied). Otherwise, since $V$ is path connected we can choose points $x_1$ and $y_1$ so that $d(x_1,y_1) < \rho$ and $y_1 \in S$ (so that \eqref{eq:equicontinuity_new_pair_2} is satisfied) and $x_1 \notin S$ (so that \eqref{eq:equicontinuity_new_pair_1} is satisfied).
 In the rest of the proof we will assume that $\delta > \lambda_1 + \lambda_2/2$.

Choose points $x_m$, $x_M$, $y_m$, and $y_M$ so that
\begin{equation*}
\begin{array}{c c}
u_{n}(x_m) = \min_{z\sim x}u_{n}(z),& u_{n}(x_M) = \max_{z\sim x}u_{n}(z),\\
u_{n}(y_m) = \min_{z\sim y}u_{n}(z),& u_{n}(y_M) = \max_{z\sim y}u_{n}(z).
\end{array}
\end{equation*}
Take a point $z_m$ such that $d(y,z_m) \leq \epsilon - d(x,y)$ and $d(z_m,y_m) < \rho$, which surely exists, since $d(x,y) < \rho$ and $d(y,y_m) \leq \epsilon$. By the triangle inequality this point satisfies $d(x,z_m) \leq \epsilon$. Therefore we have $z_m \sim x$,  $z_m \sim y$ and
\begin{equation}
  \label{eq:equicontinuity_separation}
d(y_m,\{z:z\sim x,z \sim y\}) < \rho.
\end{equation}
Analogously we construct a point $z_M$ such that $d(x_M,z_M) < \rho$ and $d(z_M,y) \leq \epsilon$. Now we  have 
\begin{equation}\label{eq:equiconotinuity_separation_2}
u_n(x_M) - u_n(y_M) \leq u_n(x_M) - u_n(z_M) \leq \lambda_2.
\end{equation}
Next calculate
\begin{align}\label{eq:equicontinuity_1}
(u_{n}(x_M) - u_{n}(y_M))& + (\min_{z \sim x, z\sim y}u_{n}(z) - u_{n}(y_m)) \nonumber \\
&\geq (u_{n}(x_M) - u_{n}(y_M)) + (u_{n}(x_m) - u_{n}(y_m)) \nonumber \\
& = 2(u_{n+1}(x)-f(x) -u_{n+1}(y) + f(y))\nonumber \\
&\geq 2\delta - 2\lambda_1.
\end{align}
Plugging \eqref{eq:equiconotinuity_separation_2} into \eqref{eq:equicontinuity_1} we get
\begin{equation}\label{eq:equaicontinuity_2}
\min_{z \sim x, z\sim y}u_{n}(z) - u_{n}(y_m) \geq 2\delta-2\lambda_1 - \lambda_2.
\end{equation}
Now define $r$ as the supremum of the values $\tilde{r}$ such that for every $z_0 \in B(y_m,\tilde{r})$ we have $u_n(z_0) < \min_{z \sim x, z\sim y}u_{n}(z)$. 
By \eqref{eq:equicontinuity_separation}, \eqref{eq:equaicontinuity_2} and the assumption on $\delta$ it follows that $r$ is well defined and $0 < r < \rho$.
Finally we take a point $x_1$ such that $d(y_m,x_1) = r$ with $u_n(x_1) = \min_{z \sim x, z\sim y}u_{n}(z)$ (which exists by the definition of $r$). By \eqref{eq:equaicontinuity_2} we have
\begin{equation*}
u_n(x_1) = \min_{z \sim x, z \sim y}u_n(z) \geq u_n(y_m) + 2\delta-2\lambda_1 - \lambda_2.
\end{equation*}
Furthermore \eqref{eq:equaicontinuity_2} also implies
\begin{equation*}
u_{n+1}(y) \geq u_n(y)-\lambda_3 \geq \min_{z \sim x, z \sim y}u_n(z) - \lambda_3 \geq  u_n(y_m) + 2\delta-2\lambda_1 - \lambda_2 - \lambda_3.
\end{equation*}
This proves the claim with $y_1=y_m$.
\end{proof}

\begin{proof}[Proof of Lemma \ref{prop:equicontinuity}]
By part (ii) of Lemma \ref{lemma:comparison} it is enough to prove the claim for an arbitrary shift of the payoff function $f$, so by the definition of $c_f$, we can assume that $c_f=0$ (see Remark \ref{rem:shifting constant scaling}).

By Lemma \ref{lemma:convergence_of_differences}, for a given $\lambda_3 > 0$ choose $n_0$ large enough so that $\|u_n-u_{n-1}\| \leq \lambda_3$, for all $n \geq n_0$. Assume that the sequence $(u_n)$ is not equicontinuous. Then there is a $\delta_0 > 0$ such that for any $\rho >0$ there are infinitely many integers $k$ satisfying $\osc(u_{n_0+k},\rho) \geq \delta_0$. Fix such a $k$ and $\rho$ and define 
$
\delta_1 =  \osc(u_{n_0+k},\rho).
$
Since $\|u_{n_0+k}-u_{n_0+\ell}\| \leq (k-\ell) \lambda_3$, for all $0 \leq \ell \leq k$ we get that 
\begin{equation}\label{eq:equicontinuity_l}
\osc(u_{n_0+\ell},\rho) \leq \delta_1  + 2(k-\ell) \lambda_3.
\end{equation}
Now fix an arbitrary $\tau$ and let $x_0$ and $y_0$ be points that satisfy $d(x_0,y_0) < \rho$ and
\begin{equation}\label{eq:equicontinuity_first_step}
u_{n_0+k}(x_0) - u_{n_0+k}(y_0) \geq \delta_1 - \tau.
\end{equation}
Applying Lemma \ref{lemma:equicontinuity_help} to the pair $(x_0,y_0)$ with $\delta=\delta_1-\tau$, $\lambda_2=\delta_1+2\lambda_3$, $\lambda_1 = \osc(f,\rho)$ and $\lambda_3$ defined as before
we obtain points $x_1$ and $y_1$ such that $d(x_1,y_1) < \rho$, and 
\[
u_{n_0+k-1}(x_1) - u_{n_0+k-1}(y_1) \geq \delta_1-2\tau-2\lambda_1-2\lambda_3,
\]
and
\[
u_{n_0+k}(y_0) - u_{n_0+k-1}(y_1) \geq \delta_1 - 2\tau -2\lambda_1 - 3\lambda_3.
\]
Using \eqref{eq:equicontinuity_l} and applying the same arguments inductively, for $\lambda_1$, $\lambda_3$ and $\tau$ small enough, we obtain a sequence of points $(x_\ell)$, $1 \leq \ell \leq k$, which satisfy the inequalities 
\begin{equation}
  \label{eq:equicontinuity_final_1}
u_{n_0+k-\ell}(x_\ell) - u_{n_0+k-\ell}(y_\ell) \geq \delta_1 - a_{\ell}\tau - b_{\ell}\lambda_1 - c_{\ell}\lambda_3,
\end{equation}
and
\begin{equation}
  \label{eq:equicontinuity_final_2}
  u_{n_0+k-\ell+1}(y_{\ell-1}) - u_{n_0+k-\ell}(y_\ell) \geq \delta_1 - a_{\ell}\tau - b_{\ell}\lambda_1 - (c_{\ell}+1)\lambda_3,
\end{equation}
where the coefficients satisfy 
$a_{1}=2$, $b_{1}=2$, $c_1=2$ and
\[
a_{\ell+1}=2a_\ell, \ b_{\ell+1}=2(b_\ell+1), \ c_{\ell+1}=2(c_\ell+\ell +1).
\]
This leads to $a_\ell=2^\ell$, $b_\ell=2^{\ell+1}-2$ and $c_\ell= 2^{\ell+2} -2\ell -4$. Summing \eqref{eq:equicontinuity_final_2} with these values of coefficients for $1 \leq \ell \leq k$ we obtain
\begin{equation}
  \label{eq:equicontinuity_final_3}
u_{n_0+k}(y_0) - u_{n_0}(y_k) \geq k\delta_1-2^k(2\tau+4\lambda_1+8\lambda_3).  
\end{equation}
Now taking $k$ to be the largest integer not larger than $\frac{\log(\delta_1/(2\tau+4\lambda_1+8\lambda_3))}{\log 2}$ and increasing the value of $n_0$ if necessary, leads to
\begin{equation}\label{eq:equicontinuity_final_4}
u_{n_0+k}(y_0) - u_{n_0}(y_k) \geq \delta_1\Big(\frac{\log(\delta_1/(2\tau+4\lambda_1+8\lambda_3))}{\log 2} -2\Big).\end{equation}
Since the values of $\tau$, $\lambda_1$ and $\lambda_3$ can be chosen arbitrarily small and $\delta_1$ is bounded from below by $\delta_0$, the right hand side of \eqref{eq:equicontinuity_final_4} can be arbitrarily large. This is a contradiction with the assumption that $c_f=0$.

\end{proof}

\begin{proof}[Proof of Theorem \ref{thm:graphs_with_loops}]
By Theorem \ref{thm:first_result},  $(u_n)$ is a bounded sequence of functions
In the case of finite graphs with loops the statement follows from Proposition \ref{lemma:subsequential_to_complete_convergence}.
For the case of adjacency graphs, note that,
by Lemma \ref{prop:equicontinuity} $(u_n)$ is also equicontinuous and, by the Arzela-Ascoli theorem, it has a convergent subsequence. Now the claim follows from Proposition \ref{lemma:subsequential_to_complete_convergence}.
\end{proof}

\begin{proof}[Proof of Theorem \ref{thm:solutions_on_general_graphs} in the case of adjacency graphs]
If $u$ is a solution to \eqref{eq:equation_discrete} then playing the game with terminal payoff $u_0=u$ and running payoff $f$ gives the constant sequence of game values $u_n=u$.
In the other direction, it is clear that the limit in Theorem \ref{thm:graphs_with_loops} is a solution to the equation \eqref{eq:equation_discrete}.
\end{proof}


\begin{example}\label{example:no_convergence}
Let $G$ be a bipartite graph with partition of the vertex set into $V_1$ and $V_2$ (meaning $V_1 \cap V_2=\emptyset$, $V = V_1\cup V_2$ and all edges in the graph are connecting vertices in $V_1$ and $V_2$). Let $f$ be a function on $V$ having value $1$ on $V_1$ and $-1$ on $V_2$. Then if $u_0 = 0$ it is easy to check from \eqref{eq:recursion} that $u_n=f$ if $n$ is odd and $u_n=0$ if $f$ is even, and therefore the sequence $(u_n)$ does not converge. However $u=f/2$  is a solution to \eqref{eq:equation_discrete}.
\end{example}

From the proof of Lemma \ref{lemma:convergence_of_differences} we can extract the following result about the speed of convergence.

\begin{proposition}\label{prop:speed_of_convergence}
There is a universal constant $C> 0$ such that, under the assumptions of Theorem \ref{thm:graphs_with_loops}, for $n \geq 2$ we have 
\begin{equation}
  \label{eq:speed_of_convergence_statement}
\|u_{n+1} - u_{n}\| \leq \frac{AC}{\log n},
\end{equation}
where $A = (\max u_0 - \min u_0) + \diam(G)^2 (\max f - \min f)$. 
\end{proposition}

\begin{proof}
Again, it is enough to prove the claim when $\|u_{n+1} - u_{n}\|$ is replaced by $M_{n}^f(u_0) = \max(u_{n} - u_{n-1})$. 
If $\max u_m < \min u_n$ for some $m<n$ then by Lemma \ref{lemma:lipschitz_properties} we have $u_{m +k(n-m)} - u_m \geq k(\min u_n - \max u_m)$ for all $k \geq 0$, which contradicts the boundedness of $(u_n)$ (which in turn follows from the assumption $c_f=0$). Similarly we get the contradiction when $\max u_m < \min u_n$ for some $n<m$.
Therefore we have
$\min u_{n} \leq \max u_m$ for all  $m$ and $n$
and Lemma \ref{lemma:boundedness_of_differences}  implies that 
\begin{equation}\label{eq:speed_of_convergence_1}
\max u_{n+k} - \min u_n \leq 2A.
\end{equation} 
By Lemmas \ref{lemma:lipschitz_properties} and \ref{lemma:convergence_of_differences} we know that $(M_n^f(u_0))$ is a non-increasing sequence converging to $0$. For given $r > \delta$ assume $n$ and $k$ are such that for all $n \leq m \leq n+k$ we have $r- \delta \leq M_m^f(u_0) \leq r$. Now \eqref{eq:going_back_1} implies that
\[
\max u_{n+k} - \min u_n \geq  kr - k\delta - 2^k\delta.
\]
Combining this with \eqref{eq:speed_of_convergence_1} we see that, if $K(r-\delta,r)$ is the number of indices $m$ such that $r-\delta \leq M_m^f(u_0)  \leq r$, then for all integers $0 \leq k \leq K(r-\delta,r)$ we have
$kr - 2^{k+1}\delta\leq 2A$.
Taking $\delta = r 2^{-2A/r-2}$ we get that $K(r-\delta,r) < 1+ 2A/r$. Now let $r_0 > 0$ and define the sequence $r_{n+1} = r_n(1-2^{-2A/r_n-2})$, which is clearly decreasing and converging to $0$. By the above discussion we have 
\begin{equation}\label{eq:speed_of_convergence_3}
K(r_{n+1} ,r_n) \leq \frac{2A}{r_n}+1.
\end{equation} 
Furthermore define $N(\alpha,\beta) = \sum K(r_{n+1},r_{n})$, where the sum is taken over all indices $n$ for which the interval $[r_{n+1},r_n]$ intersects the interval $[\alpha,\beta]$.
Defining $s_n = \log (2A/r_n)$ we have $s_{n+1} = s_n  - \log \Big(1- 2^{-e^{s_n}-2}\Big)$. Since the function $s \mapsto \log \Big(1- 2^{-e^{s}-2}\Big)$ is negative and increasing, the number of indices $n$ such that the interval $[s_n,s_{n+1}]$ intersects a given interval $[a,b]$ is no more than 
\[
\frac{b-a}{-\log\Big(1- 2^{-e^b-2}\Big)}+2 \leq (b-a)2^{e^b+2}+2,
\]
where we used the inequality $\log(1-x) \leq -x$, for $0 \leq x < 1$. 
This together with \eqref{eq:speed_of_convergence_3} implies 
\[N(2Ae^{-b}, 2Ae^{-a}) \leq \Big( (b-a)2^{e^b+2}+2 \Big) \Big( e^b + 1\Big).\]
Therefore we have 
\[
N(2Ae^{-t},2A) \leq (4+o(1)) 2^{e^t}e^t,
\]
and since  $M_1^f(u_0) \leq 2A$, there are no more than $(4+o(1))2A2^{2A/r}r^{-1}$ indices $n$ such that $M_n^f(u_0) \geq r$, which then easily implies the claim.
\end{proof}

\begin{remark}\label{rem: holds for shifts of f}
From Lemma \ref{lemma:comparison} (ii) it is clear that removing the assumption $c_f=0$ from the statements of Lemma \ref{lemma:convergence_of_differences} and Proposition \ref{prop:speed_of_convergence}
yields $\lim_n (u_{n+1}-u_n)  = c_f$ 
and
$|\|u_{n+1} - u_{n}\|-c_f| \leq \frac{AC}{\log n}$ respectively.
\end{remark}

One of the obstacles to faster convergence is the fact that for each vertex $x$ the locations  where the maximum and the minimum values of $u_n$ among its neighbors are attained depends on $n$. However, in the case of finite graphs with loops, these locations will eventually be ``stabilized'', if (for example) the limiting function is one-to-one. Therefore after a certain (and possibly very large) number of steps, we will essentially see a convergence of a certain Markov chain, which is exponentially fast. To prove this in the next theorem recall some basic facts about finite Markov chains. A time homogeneous Markov chain $X$ on a finite state space is given by its transition probabilities $P(i,j) = \mathbb{P}(X_1=j|X_0=i)$. Denote the transition probabilities in $k$  steps as $P^k(i,j)=\mathbb{P}(X_k=j|X_0=i)$ (these are just entries of the $k$th power of the matrix $(P(i,j))_{ij}$).
An \emph{essential class} of a  Markov chain is a maximal subset of the state space with the property that for any two elements  $i$ and $j$ from this set there is an integer $k$ such that $P^k(i,j)>0$. An essential class is called \emph{aperiodic} if it contains an element $i$ such that the greatest common divisor of integers $k$ satisfying $P^k(i,i)>0$ is $1$.
The state space can be decomposed into several disjoint essential classes and a set of elements $i$ which are not contained in any essential class and which necessarily satisfy $P^k(i,j)>0$ for some integer $k$ and some element $j$ contained in an essential class. If all essential classes of a Markov chain are aperiodic then the distribution of $(X_n)$ converges to a stationary distribution and, moreover this convergence is exponentially fast. This result is perhaps more standard when the chain is \emph{irreducible} (the whole state space is one essential class). However the more general version we stated is a straightforward consequence of this special case after we observe that the restriction of a Markov chain to an aperiodic essential class is an irreducible Markov chain, and that for any element $i$ not contained in any essential class, conditioned on $X_0=i$, the time of the first entry to an essential class is stochastically dominated from above by a geometric random variable. For more on this topic see \cite{LPW}.

\begin{proposition}\label{prop:fast_convergence_after_burnout}
Let $G$ be a finite graph with a loop at each vertex, $f$ a function on the set of vertices and $(u_n)$ a sequence of game values played with running payoff $f$. Assuming $c_f=0$, let $u$ be the limit of the sequence $(u_n)$ and assume that for each vertex $x \in V$ there are unique neighbors $y_m$ and $y_M$ of $x$, such that $u(y_m) = \min_{y \sim x}u(y)$ and  $u(y_M) = \max_{y \sim x}u(y)$.
 Then there are constants $C>0$ and $0<\alpha<1$ (depending on $G$, $f$ and $u_0$) such that $\|u_n - u\| \leq C\alpha^n$.
\end{proposition}
\begin{proof}
Let $A=(a_{xy})$ be the matrix such that $a_{xy}=1/2$ if either $u(y) = \max_{z \sim x}u(z)$ or $u(y) = \min_{z \sim x}u(z)$ and $0$ otherwise.
The Markov process $X_k$  on the vertex set, with the transition matrix $A$, has the property that all essential classes are aperiodic. To see this fix an essential class $I \subset V$ let $x$ be a vertex such that $u(x) = \max_Iu$, and observe that $a_{xx}=1/2$. Therefore the distribution of $X_k$ converges exponentially fast to a stationary distribution.

Since $u = \lim_n u_n$, there is an $n_0$ such that for $n \geq n_0$ and any vertex $x$ the unique neighbors of $x$ where $u$ attains the value $\max_{z \sim x}u(z)$ ($\min_{z \sim x}u(z)$) and where $u_n$ attains the value $\max_{z \sim x}u_n(z)$ ($\min_{z \sim x}u_n(z)$) are equal. Writing functions as column vectors, this means that $u_{n+1} = Au_n + f$ for $n \geq n_0$.
Thus, defining $v_n=u_{n+1} - u_n$, for $n \geq n_0$  we have 
\[
v_{n+1} = u_{n+2} - u_{n+1} = A u_{n+1} - A u_n = A v_n.
\]
This means that for any $k \geq 0$ we have $v_{n_0+k}(x) = \mathbb{E}_x(v_{n_0}(X_k))$. Therefore the sequence of functions $(v_{n_0+k})_k$ converges exponentially fast. Since we necessarily have $\lim_n v_n = 0$ the claim follows from $\|u_n - u\| \leq \sum_{k=n}^\infty \|v_k\|$.
\end{proof}

Our next goal is to prove Theorem \ref{thm:solutions_on_general_graphs} for all finite graphs. 
Recall the (nonlinear) operator $A_f$ from Lemma \ref{lemma:lipschitz_properties}.
For a real number $c \in \mathbb{R}$, and a function $u$ define $D_f(u,c) = \|A_{f - c}u-u\|$. 
To prove the existence of a solution it is  enough to prove that $D_f$ has a minimum value equal to 0. First we use a compactness argument to prove that it really has a minimum. For the rest of this section all the graphs will be arbitrary connected finite graphs.

\begin{lemma}\label{lemma: estimates}
Let $G$ be a finite connected graph, and $f$ and $u$ functions on $V$. Then 
\begin{equation}\label{eq:1}
\max u - \min u \leq 2^{\diam(G)+1}(\|f\| +D_f(u,0)).
\end{equation}
\end{lemma}

\begin{proof}
Assume that the function $u$ attains its minimum and maximum at vertices $x_m$ and $x_M$ respectively.
Let $x_m=y_0, y_1, \dots , y_{k-1}, y_k = x_M$ be a path connecting $x_m$ and $x_M$ with $k \leq \diam(G)$. 
Observe that 
$$
A_{f}u(y_i) \geq \frac{u(x_m) + u(y_{i+1})}{2} + f(y_i),
$$
for $i=0,\dots,k-1$.
Estimating the left hand side of the above equations by $A_fu \leq u + D_f(u,0)$ we get
$$
u(y_{i+1}) \leq 2u(y_i) + 2D_f(u,0) -2f(y_{i}) - u(x_m),
$$
for $i=0, \dots, k-1$.
Multiplying the $i$-th inequality by $2^{k-1-i}$, for $i=0, \dots k-1$ and adding them we obtain 
\[u(x_M) - u(x_m) \leq (2^{k+1}-2)(D_f(u,0) - \min f),
\]
which implies the claim.
\end{proof}

\begin{lemma}\label{lemma: minimum} 
Under the assumptions of Lemma \ref{lemma: estimates}
the function $D_f(u,c)$ has a minimum value.
\end{lemma}

\begin{proof}
Since $D_f$ is a continuous function, we only need to prove that $\tau := \inf D_f = \inf_{\mathcal{U} \times I}D_f$, where the right hand side is the infimum of the values of $D_f$ over $\mathcal{U} \times I$ for a bounded set of functions $\mathcal{U}$ and a bounded interval $I$. First assume that $c$ is a constant large enough so that $f+c$ has all values larger than $\tau + 1$. If $x_m$ is a vertex where a function $u$ attains its minimum, we have
\[
\frac{1}{2}(\max_{y \sim x_m}u(y) + \min_{y \sim x_m}u(y)) + f(x) + c \geq u(x_m) + \tau +1.
\] 
This implies that $D_f(u,-c) \geq \tau + 1$ for any function $u$. Similarly for sufficiently large $c$ we have that  $D_f(u,c) \geq \tau + 1$ for any function $u$.
Therefore there is a bounded interval $I$ such that the infimum of values of $D(u,c)$ over all functions $u$ and $c \notin I$ is strictly bigger than $\tau$.

Furthermore by Lemma \ref{lemma: estimates} we can find a constant $K$ such that for any $c \in I$ we have that $\max u - \min u \geq K$ implies $D_f(u,c) \geq \tau+ 1$. Also since $D_f(u + \lambda,c) = D_f(u,c)$ for any $\lambda \in \mathbb{R}$, we have that $\tau = \inf_{\mathcal{U} \times I}D_f$ where $\mathcal{U}$ is the set of functions such that $\min u =0$ and $\max u \leq K$. Since the set $\mathcal{U}$ is bounded the claim follows.
\end{proof}

\begin{proof}[Proof of Theorem \ref{thm:solutions_on_general_graphs} in the case of finite graphs]
Assuming the existence of a solution the argument proceeds  as in the proof of the adjacency case. By the same argument,
to show the other direction, it is enough to prove that there is a constant $c$ for which there is a solution to \eqref{eq:equation_discrete}, when the right hand side $f$ is replaced by $f-c$, since then we necessarily have $c=0$. In other words it is enough to show that $\min D_f =0$. 
By Lemma \ref{lemma: minimum} this minimum is achieved and denote it by $m = \min D_f$. Assume that $m>0$.
Fix a pair $(u,c)$ where the minimum is achieved and define $S_{u,c}^+ :=\{x: A_{f-c}u(x) -u(x) = m\}$,  $S_{u,c}^- :=\{x: A_{f-c}u(x) - u(x) = - m\}$ and $S_{u,c}:=S_{u,c}^+ \cup S_{u,c}^-$. By definition $S_{u,c} \neq \emptyset$. If $S_{u,c}^+ = \emptyset$ then there is a $\delta>0$ small enough so that $A_{f-c+\delta}u - u < m$, and of course  $A_{f-c+\delta}u - u > -m$. This implies that $D_f(u,c-\delta) < m$,
which is a contradiction with the assumption that $m=\min D$. Therefore $S_{u,c}^+ \neq \emptyset$, and similarly  $S_{u,c}^- \neq \emptyset$.

Call a set $S_r \subset S_{u,c}^+$ \textit{removable} for function $u$, if both of the following two conditions hold:
\begin{itemize}
\item[(i)] 
For every $x \in S_{u,c}^+$ there is a $y \notin S_r$ so that $y \sim x$ and  $u(y) = \min_{z \sim x}u(z)$.
\item[(ii)] 
There are no $x \in S_{u,c}^+$ and $y \in S_r$ so that $y \sim x$ and $u(y) = \max_{z\sim x}u(z)$.
\end{itemize}
By increasing values of the function $u$ on $S_r$ we can remove this set from $S_{u,c}^+$. More precisely, define the function $\tilde{u}_{\delta}$ so that $\tilde{u}_{\delta}(x) = u(x)$ for $x \notin S_r$ and $\tilde{u}_{\delta}(x) = u(x) + \delta$ for $x \in S_r$.
Since the graph $G$ is finite, for $\delta $ small enough and all points $x \notin S_{u,c}^+$, we have $A_{f-c}\tilde{u}_{\delta}(x) - \tilde{u}_{\delta}(x) < m$. Furthermore, by the above two conditions, if $\delta$ small enough, for any point $x \in S_{u,c}^+$ we have $A_{f-c}\tilde{u}_{\delta}(x) =A_{f-c} u(x)$. On the other hand for $x \in S_r$ we have 
\[
A_{f-c}\tilde{u}_{\delta}(x) - \tilde{u}_{\delta}(x) = m - \delta,
\]
and therefore $S_{\tilde{u}_{\delta},c}^+ = S_{u,c}^+\backslash S_r$. Moreover $S_{\tilde{u}_{\delta},c}^- \subset S_{u,c}^-$ is obvious.

Similarly we can define removable sets $S_r^-$ contained in $S_{u,c}^-$ so that there are no $x \in S_{u,c}^-$ and $y \in S_r$ such that $u(y) = \min_{z\sim x}u(z)$ and that for every $x \in S_{u,c}^-$ there is a $y \notin S_r$ such that $u(y) = \max_{z \sim x}u(z)$. This set can be removed from $S_{u,c}^-$ be decreasing the value of $u$ on this set. 
Note that the removable sets in $S_{u,c}^+$ and $S_{u,c}^-$ can be removed simultaneously as described above.
Thus if a pair $(u,c)$ minimizes the value of $D_f$, and $\tilde{u}$ is obtained from $u$ by removing removable sets in $S_{u,c}^+$ and $S_{u,c}^-$,
then the pair $(\tilde{u},c)$ also minimizes the value of $D_f$, and moreover $S_{\tilde{u},c} \subset S_{u,c}$.

Call a function $u$ \textit{tight} (for $f$) if there is $c \in \mathbb{R}$ such that the pair $(u,c)$ minimizes $D_f$, and so that the set $S_{u,c}$ is of smallest cardinality, among all minimizers of $D_f$. By the discussion above, tight functions have no non-empty removable sets. For a tight function $u$ define $v =A_{f-c}u$. By Lemma \ref{lemma:lipschitz_properties} we have that $D_f(v,c)= \|A_{f-c}v -v\| \leq m$ and because $m=\min D_f$ we have $D_f(v,c) = m$.

Now observe that it is enough to prove that for any tight function $u$ and $v=A_{f-c}u$, the set $S_{v,c}^+ \backslash S_{u,c}^+$ is removable for function $v$. To see this first note that by symmetry the set $S_{v,c}^- \backslash S_{u,c}^-$ is also removable for $v$. Let $v_1$ be a function obtained by removing all these vertices as described above. In particular we have $v_1(x) = v(x) = u(x) + m$ for $x \in S_{u,c}^+$ and $v_1(x) = v(x)= u(x) - m$ for $x \in S_{u,c}^-$.
The function $v_1$ then satisfies $S_{v_1,c}^+ \subseteq S_{u,c}^+$ and $S_{v_1,c}^- \subseteq S_{u,c}^-$. By tightness of $u$ it follows that $S_{v_1,c}^+ =S_{u,c}^+$ and $S_{v_1,c}^- = S_{u,c}^-$ and thus the function $v_1$ is also tight. Now we can repeat this argument to obtain a sequence of tight functions $(v_k)$ such that $S_{v_k,c}^+ = S_{u,c}^+$, $S_{v_k,c}^- = S_{u,c}^-$, $v_k(x) = u(x) + km$ for $x \in S_{u,c}^+$ and $v_k(x) = u(x) - km$ for $x \in S_{u,c}^-$. Since $D(v_k,c) = m$ for all $k$ and $\lim_k (\max v_k - \min v_k) = \infty$, Lemma \ref{lemma: estimates} gives a contradiction with the assumption that $m > 0$. 

Thus it is only left to prove that for any tight function $u$ and $v=A_{f-c}u$ the set $S_{v,c}^+ \backslash S_{u,c}^+$ is removable for function $v$. 
For this we need to check the conditions (i) and (ii) from the definition of the removable sets. 
Take a vertex $x \in S_{v,c}^+$ and note that since $A_{f-c}v(x) -v(x) = \max(v-u)$ and $v = A_{f-c}u$, by Lemma \ref{lemma:lipschitz_properties} for a $y_1 \sim x$ such that  $u(y_1) = \min_{z \sim x}u(z)$, we have $v(y_1) = \min_{z \sim x}v(z)$ and $y_1 \in S_{u,c}^+$ which checks the first assumption.
Furthermore by Lemma \ref{lemma:lipschitz_properties} for any $y_2$ such that $v(y_2) = \max_{z \sim x}v(z)$ we have $y_2 \in S_{u,c}^+$ which also checks the second assumption in the definition of removable sets.
\end{proof}

Next we present two  examples for which we  explicitly calculate the value of the  Player I's long term advantage $c_f$.

\begin{example}\label{ex:complete graph}
If $G$ is a complete graph with loops at each vertex and $f$ a function on the set of vertices, then $c_f= (\max f + \min f)/2$. To see this, use the fact that $c_f$ defined as above satisfies $c_{f+\lambda} = c_f + \lambda$ for any $\lambda \in \mathbb{R}$,  and that  $c_f = 0$ implies that $u=f$ solves \eqref{eq:equation_discrete}. 

When $G$ is a complete graph without loops the situation becomes more complicated. If the function $f$ attains both the maximum and the minimum values at more than one vertex then again we have $c_f = (\max f+ \min f) /2$, and again in the case  $\max f + \min f = 0$ the function $u=f$ satisfies equation \eqref{eq:equation_discrete}.

If the maximum and the minimum values of $f$ are attained at unique vertices then
\begin{equation}\label{eq:complete graph shifting}
c_f = \frac{\max f +\min f}{3} + \frac{\max_2 f + \min_2 f}{6},
\end{equation}
where $\max_2 f$ and $\min_2 f$ denote the second largest and the second smallest values of the function $f$ respectively.  
To prove this assume the expression in \eqref{eq:complete graph shifting} is equal to zero, and let $x_M$ and $x_m$ be the vertices where $f$ attains the maximum and the minimum value. Then define a function $u$ so that $u(x_M) = (2\max f + \max_2 f) /3$, $u(x_m) = (2\min f + \min_2 f) /3$ and
 $u(x) = f(x)$, for $x \notin \{x_m, x_M\}$. Now using the fact that $c_f=0$ and that $u$ attains its maximum and minimum values only at $x_M$ and $x_m$ respectively, it can be checked that $u$ solves \eqref{eq:equation_discrete}.

Finally in the case when the maximum value of the function $f$ is attained at a unique vertex $x_M$ and the minimum at more than one vertex we have $c_f = (2\max f + \max_2 f + 3\min f)/6$. When this expression is equal to zero, one solution $u$ of the equation \eqref{eq:equation_discrete} is given by $u(x) = f(x)$ for $x \neq x_M$ and $u(x_M) = (2\max f + \max _2 f) /3$. Similarly when the maximum of $f$ is attained at more than one vertex and minimum at a unique vertex we have  $c_f = (2\min f + \min_2 f + 3\max f)/6$.
\end{example}

\begin{example}\label{ex:linear graph}
Consider a linear graph of length $n$ with loops at every vertex, that is take $V=\{1, \dots, n\}$ and connect two vertices if they are at Euclidean distance $0$ or $1$. Let $f$ be a non-decreasing function on the set of vertices, that is $f(i) \leq f(i+1)$, for $1 \leq i \leq n-1$. By induction and \eqref{eq:recursion}, running the game with the vanishing terminal payoff and the running payoff $f$ gives sequence of game values $(u_n)$, each of which is a non-decreasing function. Representing functions $u_n$ as column vectors, we have $u_{n+1} = Au_n + f$, where $A=(a_{ij})$ is a matrix with $a_{11}=a_{nn}=1/2$, $a_{ij}=1/2$, if $|i-j|=1$ and $a_{ij}=0$ for all other values of $i$ and $j$. Therefore $v_n=u_{n+1}-u_n$ satisfies $v_{n+1} = A v_n$. Using this we see that $v_n(x)= \mathbb{E}_x(f(X_n))$, where $X_n$ is the simple random walk on the graph with the vertex set $\{1,\dots,n\}$ where $i$ and $j$ are connected with an edge if $|i-j|=1$ and with loops at $1$ and $n$. The stationary distribution of the random walk $(X_n)$ is uniform on $\{1,\dots, n\}$, and this is the limit of the distributions of $X_n$ as $n$ tends to infinity. From here it is clear that $\lim_n v_n(x) = \Big(\sum_{i=1}^n f(i)\Big) /n$ and $c_f$ is equal to the average of the values of function $f$. 

The condition that $f$ is monotone is necessary. Consider for example the linear graph with loops and three vertices and the function $(f(1),f(2),f(3))=(-1,2,-1)$. Then by \eqref{eq:recursion} we have $u_0=0$, $u_1=f$ and $u_2 = f + 1/2$ which implies that $u_{n+1} = u_n +  1/2$ for all $n \geq 1$ and by Lemma \ref{lemma:comparison} (i) we have $c_f=1/2$. 
\end{example}

\section{The continuous case}\label{section:the continuous case}

The main goal of this section is to study the game values on Euclidean $\epsilon$-adjacency graphs, as defined in the Section 1, to obtain the existence of viscosity solutions to the equation \eqref{eq:equation_continuous}. One of the main concerns will be the dependence of the game values and limits, obtained in the previous section, on values of step sizes $\epsilon$. 
The following example shows that the issue starts already with the Player I's long term advantage $c_f(\epsilon)$ (recall that $c_f(\epsilon)$ was defined as the Player I's long term advantage for a game played on an $\epsilon$-adjacency graph with the running payoff $f$).

\begin{example}\label{example:shifting_parameters_continuous}
  This example shows that, in general, for Euclidean $\epsilon$-adjacency graphs on a domain $\Omega$ and a continuous function $f \colon \overline{\Omega} \to \mathbb{R}$, the value of $c_f(\epsilon)$ depends on  $\epsilon$. First observe a trivial fact that for any $\Omega $ of diameter $\text{diam}(\Omega)$ and $f$ we have $c_f(\diam (\Omega)) = (\max f + \min f) /2$. Next let $\Omega=(0,1)$ and let $f$ be a piecewise linear function that is linear on the intervals $[0,1/2]$ and $[1/2,1]$ and has values $f(0)=f(1/2)=1$ and $f(1)=-1$. By the above observation we have $c_f(1)=0$. However notice that from \eqref{eq:recursion} it is clear that, for any $\epsilon$, playing the game with step size $\epsilon$, the vanishing terminal payoff and the running payoff $f$, the game values will be non-increasing functions on $[0,1]$. Therefore in the game of step size $1/2$ the game values $u_n$ at points $0$, $1/2$ and $1$ are equal to the game values played on the linear graph with three vertices and loops on each vertex, with  terminal payoff  zero and running payoff equal to $1$, $1$ and $-1$ at the leftmost, central and the rightmost vertex respectively. 
Using Example \ref{ex:linear graph} and going back to the game on $[0,1]$ this implies that $c_f(1/2) = 1/3$.

For a more comprehensive example, construct a monotone function $f$ on $[0,1]$ such that no value of $(2^n+1)^{-1}\sum_{k=0}^{2^n}f(k2^{-n})$ is  attained for two distinct integers $n$. By the above reasoning and Example \ref{ex:linear graph} the value of $c_f(\epsilon)$ varies for arbitrarily small values of $\epsilon$.
\end{example}

For the remainder of this paper, all the graphs are assumed to be $\epsilon$-adjacency graphs and the dependence on $\epsilon$ will be explicitly specified.

Theorem \ref{thm:convergence_of_shifting_constants} settles the issue raised in the above example. 
We will need several technical lemmas for the proof of Theorem \ref{thm:convergence_of_shifting_constants}.
The main ingredient of the proof is a comparison between the values of discrete infinity Laplacian with different step sizes from Lemma \ref{lemma:comparison_of_steps}. The idea for (as well as one part of) this lemma came from \cite{AS}.

\begin{lemma}\label{lemma:starting_from_the_subsolution}
If  $f$ and $u$ are continuous functions on a compact length space $V$ such that $-\Delta_\infty^\epsilon u \leq f$ then $c_f(\epsilon) \geq 0$. Similarly  $-\Delta_\infty^\epsilon u \geq f$ implies $c_f(\epsilon) \leq 0$.
\end{lemma}

\begin{proof}
The second claim follows by replacing $u$ and $f$ by $-u$ and $-f$ respectively, so it is enough to prove the first one.
The condition $-\Delta_\infty^\epsilon u \leq f$ can be rewritten as 
\[
\frac{1}{2}\Big(\max_{z \sim x}u(z)+ \min_{z \sim x}u(z)\Big) + \frac{\epsilon^2}{2} f(x) \geq u(x).
\]
Thus the game value $u_1$ of the first step of the game, played with the terminal payoff $u_0=u$, running payoff $\epsilon^2 f/2$ and step sizes $\epsilon$ satisfies $u_1 \geq u_0$. By  Lemma \ref{lemma:lipschitz_properties} we have $u_{n+1} \geq u_n$ for any $n$, hence $c_f(\epsilon) \geq 0$ is clear.
\end{proof}

\begin{lemma}\label{lemma:continuity_of_shiftings_in_step_sizes}
Mapping $\epsilon \mapsto c_f(\epsilon)$ is continuous on $\mathbb{R}^+$.
\end{lemma}

\begin{proof}
For a given $\epsilon$ let $u_\epsilon \in C(V)$ be a solution of $-\Delta_\infty^\epsilon u_\epsilon = f - c_f(\epsilon)$, which exists by Theorem \ref{thm:solutions_on_general_graphs} and Remarks \ref{rem:two_laplacians} and \ref{rem:shifting constant scaling}. 
Since for any $u \in C(V)$, it holds that $\epsilon \mapsto -\Delta_\infty^\epsilon u$ is a continuous function from $\mathbb{R}^+$ to $(C(V),\|\cdot \|_\infty)$, so
for a fixed $\epsilon$ and any $\delta>0$ we can find $\eta>0$ such that $|-\Delta_\infty^{\epsilon_1} u_{\epsilon}| \leq f-c_f(\epsilon) + \delta$ whenever $|\epsilon_1 - \epsilon| \leq \eta$. Now by applying Lemma \ref{lemma:starting_from_the_subsolution} we see that for such $\epsilon_1$ we have $|c_f(\epsilon_1) -c_f(\epsilon)| \leq \delta$, which gives the continuity.
\end{proof}

As mentioned above, the main part of the proof of Theorem \ref{thm:convergence_of_shifting_constants} is contained in the following lemma. For Euclidean $\epsilon$-adjacency graphs, the first inequality in \eqref{eq:comparison_of_steps_1} already appeared as Lemma 4.1 in \cite{AS}. However since their definition of the discrete Laplacian was somewhat different close to the boundary $\partial \Omega$, their estimates held only away from $\partial \Omega$. This issue does not appear in our case and their proof goes verbatim. For reader's convenience we repeat their proof of the first inequality in \eqref{eq:comparison_of_steps_1}. 

For a function $u \colon V \to \mathbb{R}$ we first define $\overline{u}^\epsilon= \max_{z \in B(x,\epsilon)}u(z)$ and $\underline{u}_\epsilon = \min_{z \in B(x,\epsilon)}u(z)$. Furthermore define $T_\epsilon^+u(x) = \overline{u}^\epsilon(x) - u(x)$ and $T_\epsilon^-u(x) = u(x) - \underline{u}_\epsilon(x)$ (this corresponds to $\epsilon S_\epsilon^+$ and $\epsilon S_\epsilon^-$ in \cite{AS}).
 Now we can write $-\Delta_\infty^\epsilon u = (T_\epsilon^-u - T_\epsilon^+u)/\epsilon^2$.

\begin{lemma}\label{lemma:comparison_of_steps}
Suppose that $u \in C(V)$ satisfies $-\Delta_\infty^\epsilon u \leq f_1$ for some $f_1\in C(V)$. Then we have 
\begin{equation}\label{eq:comparison_of_steps_1}
-\Delta_\infty^{2\epsilon} \overline{u}^\epsilon \leq \overline{f_1}^{2\epsilon} \text{ and } -\Delta_\infty^{\epsilon} \overline{u}^\epsilon \leq \overline{f_1}^\epsilon.
\end{equation}
If, in addition, we have $-\Delta_\infty^{2\epsilon} u \leq f_2$ for some $f_2 \in C(V)$ then
\begin{equation}\label{eq:comparison_of_steps_2}
-\Delta_\infty^{3\epsilon} \overline{u}^\epsilon \leq (8\overline{f_2}^{2\epsilon}  + \overline{f_1}^{\epsilon})/9.
\end{equation}
\end{lemma}

\begin{proof}
In the proof we will repeatedly use the following arguments. If $z_0,z_1 \in V$ are such that $z_1 \in B(z_0,\delta)$ and $v(z_1) = \overline{v}^\delta(z_0)$ then we have
\begin{equation}
  \label{eq:comparison_of_steps_help_1}
  T_\delta^+v(z_0) = v(z_1) - v(z_0) \leq T_\delta^-v(z_1).
\end{equation}
Furthermore the assumption $-\Delta_\infty^\delta v \leq f$ implies that 
\begin{equation}
  \label{eq:comparison_of_steps_help_2}
  T_\delta^+v(z_0) \leq T_\delta^-v(z_1) \leq T_\delta^+v(z_1) + \delta^2 f(z_1).
\end{equation}

Denote points $y_1 \in B(x,\epsilon)$, $y_2 \in B(y_1,\epsilon)$, $z_M \in B(x,2\epsilon)$ and $z_m \in B(x,2\epsilon)$ so that
$u(y_1) = \overline{u}^\epsilon(x)$, $u(y_2) = \overline{u}^\epsilon(y_1)$, $u(z_M) = \overline{u}^{2\epsilon}(x)$ and $u(z_m) = \underline{u}_{2\epsilon}(x)$. We calculate


\begin{align}\label{eq:comparison_of_steps_Armstrong_Smart_+}
T_{2\epsilon}^+\overline{u}^\epsilon(x)  & = \overline{u}^{3\epsilon}(x) - \overline{u}^\epsilon(x) \nonumber \\ \nonumber
& = (\overline{u}^{3\epsilon}(x) - u(y_2)) + (u(y_2) - u(y_1)) \\ \nonumber
& \geq T_\epsilon^+u(y_2) + T_\epsilon^+u(y_1) \\ \nonumber
& \geq 2 T_\epsilon^+u(y_1) - \epsilon^2f_1(y_2) \\ \nonumber
& \geq 2 T_\epsilon^+u(x) - \epsilon^2(f_1(y_2)+2f_1(y_1)) \\ 
& \geq T_\epsilon^+u(x) + T_\epsilon^-u(x) - \epsilon^2(f_1(y_2)+2f_1(y_1)+f_1(x)).
\end{align}
In the first inequality we used the fact that $B(y_2,\epsilon) \subset B(x,3\epsilon)$ and
in the second inequality we used \eqref{eq:comparison_of_steps_help_2} with $z_0 = y_1$, $z_1=y_2$ and $\delta=\epsilon$. In the next line we again used \eqref{eq:comparison_of_steps_help_2} with $z_0=x$, $z_1=y_1$ and $\delta=\epsilon$, and in the last line the assumption $-\Delta_\infty^\epsilon u \leq f_1$.

Furthermore we have
\begin{align}\label{eq:comparison_of_steps_Armstrong_Smart_-}
T_{2\epsilon}^- \overline{u}^\epsilon (x) & = \overline{u}^\epsilon(x) - \min_{B(x,2\epsilon)}\overline{u}^\epsilon \nonumber \\ \nonumber
& \leq (\overline{u}^\epsilon(x) - u(x)) + (u(x) - \underline{u}_\epsilon(x)) \\ 
& = T_\epsilon^+u(x) + T_\epsilon^-u(x).
\end{align}
The inequality above follows from the fact that for every $z \in B(x,2\epsilon)$ we have $\max_{B(z,\epsilon)} u \geq \min_{B(x,\epsilon)}u$.
Now the first inequality in \eqref{eq:comparison_of_steps_1} is obtained by subtracting \eqref{eq:comparison_of_steps_Armstrong_Smart_+} from \eqref{eq:comparison_of_steps_Armstrong_Smart_-}.

For the second inequality in \eqref{eq:comparison_of_steps_1} note that 
\[
T_\epsilon^+\overline{u}^\epsilon(x) = \overline{u}^{2\epsilon}(x) - u(y_1) \geq T_\epsilon^+u(y_1),
\]
and
\[
T_\epsilon^- \overline{u}^\epsilon(x) = \overline{u}^\epsilon(x) - \min_{B(x,\epsilon)}\overline{u}^\epsilon  \leq u(y_1)-u(x) \leq T_\epsilon^-u(y_1).
\]
Subtracting the above inequalities it follows that $-\Delta_\infty^\epsilon \overline{u}^\epsilon(x) \leq -\Delta_\infty^\epsilon u(y_1) \leq f_1(y_1) \leq \overline{f}_1^\epsilon(x)$.


We prove the inequality \eqref{eq:comparison_of_steps_2} similarly.
First we calculate
\begin{align*}
T_{3\epsilon}^+\overline{u}^\epsilon(x) & = \overline{u}^{4\epsilon}(x) - \overline{u}^\epsilon(x) \\
& = (\overline{u}^{4\epsilon}(x) - u(z_M)) + (u(z_M)-u(y_2)) + (u(y_2)-u(y_1)) \\
& \geq T_{2\epsilon}^+u(z_M) + T_\epsilon^+u(y_1) \\
& \geq T_{2\epsilon}^+u(z_M) + T_\epsilon^-u(y_1) - \epsilon^2 f_1(y_1) \\
& \geq T_{2\epsilon}^+u(z_M) + T_\epsilon^+u(x) - \epsilon^2 f_1(y_1).
\end{align*}
In the third line we used the fact that $y_2 \in B(x,2\epsilon)$ which implies that $u(y_2) \leq u(z_M)$, in the fourth line the assumption and in the last line \eqref{eq:comparison_of_steps_help_1}.

Using similar arguments again we have
\begin{align*}
T_{3\epsilon}^-\overline{u}^\epsilon(x) & =  \overline{u}^\epsilon(x) - \min_{B(x,3\epsilon)} \overline{u}^\epsilon \\
& \leq \overline{u}^\epsilon(x) - \underline{u}_{2\epsilon}(x) \\
& = (u(y_1) -u(x)) + (u(x) - u(z_m))\\
& = T_\epsilon^+u(x) + T_{2\epsilon}^-u(x)\\
& \leq  T_\epsilon^+u(x) + T_{2\epsilon}^+u(x) + (2\epsilon)^2f_2(x) \\
& \leq  T_\epsilon^+u(x) + T_{2\epsilon}^-u(z_M) + (2\epsilon)^2f_2(x).
\end{align*}

Now subtracting the above calculations and dividing by $(3\epsilon)^2$, we obtain 
\[
-\Delta_\infty^{3\epsilon} \overline{u}^\epsilon (x)\leq (T_{2\epsilon}^-u(z_M) - T_{2\epsilon}^+u(z_M))/(9\epsilon^2) + 4f_2(x)/9 + f_1(y_1)/9,
\]
from where the fact follows directly.
\end{proof}

Recall the notation  $\osc(f,\delta) = \sup_{d(x,y) \leq \delta}|f(x)-f(y)|$.

\begin{proposition}\label{prop: shifting constant estimates}
For any $f \in C(V)$, $\epsilon > 0$ and any positive integer $n$ we have
\begin{equation}\label{eq: shifting constant estimates}
\max\{|c_f(\epsilon 2^{-n}) - c_f(\epsilon)|, |c_f(\epsilon 3^{-n}) - c_f(\epsilon)| \} \leq \osc(f,2\epsilon).
\end{equation}
\end{proposition}

\begin{proof}

First note that the functions $\overline{f}^{r}$ and $\underline{f}_{r}$ differ from $f$ by at most $\osc(f,r)$ at any point, which easily implies 
\begin{equation}\label{eq:expanding for shifting constants}
\max\{|c_f(\rho) -c_{\overline{f}^{r}}(\rho)|, |c_f(\rho) -c_{\underline{f}_{r}}(\rho)|\} \leq \osc(f,r),
\end{equation} 
for any $\rho$.

Taking $u$ to be a continuous function such that $-\Delta_\infty^\delta u = f - c_f(\delta)$, by Lemma \ref{lemma:comparison_of_steps} we have $-\Delta_\infty^{2\delta} \overline{u}^\delta \leq \overline{f}^{2\delta} - c_f(\delta)$ and therefore also $-\Delta_\infty^{3\delta}\overline{u}^{2\delta} \leq \overline{f}^{4\delta} - c_f(\delta)$. By Lemma \ref{lemma:starting_from_the_subsolution} these inequalities and their symmetric counterparts imply that
\[
c_{\underline{f}_{2\delta}}(2\delta) \leq  c_f(\delta) \leq c_{\overline{f}^{2\delta}}(2\delta)\ \text{ and } \
c_{\underline{f}_{4\delta}}(3\delta) \leq  c_f(\delta) \leq c_{\overline{f}^{4\delta}}(3\delta).
\]
Applying these estimates inductively to $\delta = \epsilon 2^{-n}, \dots \epsilon/2$ and $\delta = \epsilon 3^{-n}, \dots \epsilon/3$ respectively, we see that
\[
c_{\underline{f}_{2\epsilon}}(\epsilon) \leq  c_f(\epsilon 2^{-n}) \leq c_{\overline{f}^{2\epsilon}}(\epsilon)\ \text{ and } \
c_{\underline{f}_{2\epsilon}}(\epsilon) \leq  c_f(\epsilon 3^{-n}) \leq c_{\overline{f}^{2\epsilon}}(\epsilon).
\]
Using \eqref{eq:expanding for shifting constants} with $r=2\epsilon$ and $\rho= 2\epsilon$,
these inequalities imply \eqref{eq: shifting constant estimates}. 

\end{proof}


\begin{proof}[Proof of Theorem \ref{thm:convergence_of_shifting_constants}]
Since $\min f \leq c_f(\epsilon) \leq \max f$, there are accumulation points of $c_f(\epsilon)$ as $\epsilon \downarrow 0$, and we only need to prove that there is only one. Suppose that there are two such accumulation points $c_1 < c_2$ and denote $\delta=c_2-c_1$. Let
 $I_1$ and $I_2$ be disjoint open intervals of length $\delta/2$, centered around $c_1$ and $c_2$ respectively. Let $\epsilon_0$ be a positive real number such that $\osc(f,\epsilon_0) \leq \delta/4$, and consider the open sets $J_1$ and $J_2$ defined as $J_i = c_f^{-1}(I_i) \cap (0,\epsilon_0/2)$. 
First note that the set $\{2^{-m}3^n: m,n \in \mathbb{Z}^+\}$ is dense in $\mathbb{R}^+$. This follows from the fact that $\{ n \log 3 - m\log 2: m,n \in \mathbb{Z}^+\}$ is dense in $\mathbb{R}$, which in turn follows from the fact that $\log 3/\log 2$ is an irrational number. 
Take an arbitrary $t \in J_1$ and, since $\{s/t: s \in J_2\}$ is an open set in $\mathbb{R}^+$,
we can find non-negative integers $m_0$ and $n_0$ such that $3^{n_0}2^{-m_0}t \in J_2$. Therefore 
\[
|c_f(t) - c_f(3^{n_0}2^{-m_0}t)| > \delta/2.
\]
However this gives a contradiction, since both $t$ and $3^{n_0}2^{-m_0}t$ lie in the interval $(0,\epsilon_0/2)$, and so by Proposition \ref{prop: shifting constant estimates} we have 
\[
\max\{|c_f(2^{-m_0} t) - c_f(t)|, |c_f(3^{n_0}2^{-m_0} t) - c_f(2^{-m_0}t)|\} \leq \osc(f,\epsilon_0) \leq \delta/4.
 \]
%
\end{proof}

\begin{proposition}\label{prop: equicontinuity and boundedness for varying balls}
For a sequence  $(\epsilon_n)$ converging to zero, let $(u_n)$ be a sequence of continuous functions on a compact length space $V$, satisfying
$-\Delta_\infty^{\epsilon_n} u_{n} = f-c(\epsilon_n)$. Then  $(u_{n})$ is an equicontinuous sequence and for all $n$ large enough we have
\[
\max u_n - \min u_n \leq 6\diam(V)^2 \|f\|.
\]
Furthermore,
 any subsequential limit of the sequence $(u_n)$ is Lipshitz continuous, with the Lipshitz constant $5\diam(V) \|f\|$.
\end{proposition}

\begin{proof}
It is enough to prove that for  $n$ large enough  and any $x \in \overline{V}$, we have that 
\begin{equation}\label{eq: equicontinuity 2 sufficient}
T_{\epsilon_n}^+u_n(x) \leq 5\diam(V) \|f\|\epsilon_n.
\end{equation} 
This is because, for any two points $x,y \in V$ and $n$ such that $\epsilon_n < d(x,y)$ there are points $x = x_0, x_1, \dots , x_k,x_{k+1}=y$ in $V$ such that $d(x_{i},x_{i+1}) < \epsilon_n$ and
$k=\lfloor d(x,y)/\epsilon_n \rfloor$.
Assuming that \eqref{eq: equicontinuity 2 sufficient} holds 
 we have that 
\begin{equation}\label{eq:marching_argument_0}
u_n(y) -u_n(x) = \sum_{i=0}^k(u_n(x_{i+1})-u_n(x_i)) \leq \sum_{i=0}^k T_{\epsilon_n}^+u_n(x_i) \leq Kd(x,y) + K\epsilon_n,
\end{equation}
where $K=5\diam(V) \|f\|$.
On the other hand for $\epsilon_n \geq d(x,y)$ we have $u_{n}(y) - u_{n}(x) \leq K\epsilon_n$. These two facts then easily imply the equicontinuity, the required bound on $\max u_{n} - \min u_{n}$, for $n$ large enough and the Lipshitz continuity of subsequential limits.


The rest of the proof will be devoted to establishing the bound in \eqref{eq: equicontinuity 2 sufficient}. 
We will use the ``marching argument'' of Armstrong and Smart from Lemma 3.9 in \cite{AS}.
First by \eqref{eq:comparison_of_steps_help_2} if $y \in B(x,\epsilon_n)$ is such that $u_n(y) = \overline{u_n}^{\epsilon_n}(x)$ then using the fact that $\min f \leq c_f(\epsilon_n) \leq \max f$ we have
\begin{equation}\label{eq: marching argument_1}
T_{\epsilon_n}^+u_n(x) \leq  T_{\epsilon_n}^+u_n(y) + \epsilon_n^2\|f- c_f(\epsilon_n)\| \leq T_{\epsilon_n}^+u_n(y) + 2\epsilon_n^2\|f\|.
\end{equation}
For a fixed $n$ let  $x_0 \in \overline{V}$ be a point where the value of $T_{\epsilon_n}^+u_n$ is maximized (it's a continuous function so it can be maximized) and let $M_{n} = T_{\epsilon_n}^+u_{n}(x_0)$ be the maximal value. 
Using the same argument as in \eqref{eq:marching_argument_0} and the fact that $V$ is bounded we have
\begin{equation}\label{eq: marching argument_2}
u_{n}(y)-u_{n}(x) \leq \Big(\frac{d(x,y)}{\epsilon_n} + 1\Big)M_{n}.
\end{equation}
Then for any $k$ let $x_{k+1} \in B(x_k,\epsilon_n)$ be such that $u_{n}(x_{k+1}) = \overline{u_n}^{\epsilon_n}(x_k)$. By (\ref{eq: marching argument_1}) we have that $T_{\epsilon_n}^+u_n(x_{k+1}) \geq T_{\epsilon_n}^+u_n(x_k) - 2{\epsilon_n}^2\|f\|$ and thus $T_{\epsilon_n}^+u_n(x_{k}) \geq T_{\epsilon_n}^+u_n(x_0) - 2k\epsilon_n^2\|f\|$ which implies that for any $m \geq 1$ 
\[
u_n(x_m) - u_n(x_0) = \sum_{k=0}^{m-1}T_{\epsilon_n}^+u_n(x_k) \geq mT_{\epsilon_n}^+u_n(x_0) - m^2\epsilon_n^2\|f\|. 
\] 
Combining this with (\ref{eq: marching argument_2}) we obtain that
\[
mM_n - m^2\epsilon_n^2\|f\| \leq \Big(\frac{\diam (V)}{\epsilon_n}+1\Big)M_n,
\]
which gives
\[
M_n \leq \frac{m^2\epsilon_n^2\|f\|}{m - 1 -\diam(V)/\epsilon_n}.
\]
Plugging in $m= \left \lfloor2\diam(V)/\epsilon_n +2 \right \rfloor$ proves \eqref{eq: equicontinuity 2 sufficient} for $\epsilon_n$ small enough.
\end{proof}

\begin{proof}[Proof of Theorem \ref{thm:general_continuous_solutions}]
The claim follows directly from Proposition \ref{prop: equicontinuity and boundedness for varying balls} using the Arzela-Ascoli theorem.
\end{proof}

Finally Theorem \ref{thm: convergence for different shiftings} below proves Theorem \ref{thm:existence_of_continuous_solutions}. However we will first need to state an auxiliary result which appeared as Lemma 4.2 in \cite{AS}.  For $x \in \mathbb{R}^d$ and $\epsilon > 0$, define $\mathbf{B}(x,\epsilon)$ as the closed ball around $x$ of Euclidean radius $\epsilon$. Also define the discrete infinity Laplacian $\widetilde{\Delta}_\infty^\epsilon$ on the whole $\mathbb{R}^d$ as 
 \[
 \widetilde{\Delta}_\infty^\epsilon v(x) = \frac{1}{\epsilon^2}\Big(\max_{\mathbf{B}(x,\epsilon)}v + \min_{\mathbf{B}(x,\epsilon)}v - 2v(x)\Big).
 \]
The first part of the following lemma is the content of Lemma 4.2 in \cite{AS}, while the second part is contained in its proof (see (4.5) in \cite{AS}). 

\begin{lemma}[Lemma 4.2 and (4.5) from \cite{AS}]\label{lemma:comparison_discrete_to_infinite}
For any open set $U$, function $\varphi \in C^3(U)$ and $\epsilon_0>0$ there is a constant $C>0$, depending only on $\varphi$, such that the following holds. 
\begin{itemize}
\item[(i)] For  any point $x \in U$ that satisfies $\mathbf{B}(x,2\epsilon_0) \subseteq U$ and $\nabla\varphi(x) \neq 0$ we have
\[
-\Delta_\infty \varphi(x)  \leq -\Delta_\infty^\epsilon \varphi(x) + C(1+|\nabla\varphi(x)|^{-1})\epsilon,
\]
for all $0 < \epsilon \leq \epsilon_0$. 
\item[(ii)] For any  $0 < \epsilon \leq \epsilon_0$ if $\mathbf{v}= \nabla\varphi(x)/|\nabla\varphi(x)|$ and $\mathbf{w}\in \mathbf{B}(0,1)$ is such that $\varphi(x+\epsilon \mathbf{w}) = \max_{\mathbf{B}(x,\epsilon)}\varphi$, then $|\mathbf{v}-\mathbf{w}| \leq C|\nabla\varphi(x)|^{-1}\epsilon$.
\end{itemize}
\end{lemma}

Next we give an auxiliary Lemma needed for the proof of Theorem \ref{thm: convergence for different shiftings}. First define a cone in $\mathbb{R}^d$ with vertex $x$, direction $\mathbf{v} \in \mathbb{R}^d$, $|\mathbf{v}| =1$, angle $2 \arcsin \alpha$ and radius $r$ as
\[
\mathbf{C}(x,\mathbf{v},\alpha,r)=\Big\{\lambda\mathbf{w} \in \mathbb{R}^d: 0 \leq \lambda \leq r, |\mathbf{w}|=1, \mathbf{w} \cdot \mathbf{v} \geq 1-\alpha\Big\}.
\]

\begin{lemma}\label{lemma: cones}
Let $\Omega$ be a domain with $C^1$ boundary $\partial \Omega$, let $x_0 \in \partial \Omega$. Assume $\varphi \in C^\infty(\overline{\Omega})$ is a smooth function on the closure of $\Omega$ and $\nabla_{\nu} \varphi (x_0) > 0$. Then we can find positive $\alpha$ and $r$ and an open set $U$ containing $x_0$ such that $\mathbf{C}(x,-\nabla\varphi(x)/|\nabla\varphi(x)|,\alpha,r) \subset \overline{\Omega}$, for all $x \in U \cap \overline{\Omega}$.
\end{lemma}

\begin{proof}
Denote $\mathbf{v} = \nabla\varphi(x_0)/|\nabla\varphi(x_0)|$.
By the continuity of $\nabla \varphi$ it is enough to prove that for some $\alpha$ and $r$ and $U$ we have $\mathbf{C}_x = \mathbf{C}(x, -\mathbf{v},\alpha,r) \subset \overline{\Omega}$, for all $x \in U \cap \overline{\Omega}$. Moreover define the reverse cone $\mathbf{C}'_x=  \mathbf{C}(x,\mathbf{v},\alpha,r)$.
It is clear that we can find $\alpha$ and $r$ satisfying 
$\mathbf{C}_{x_0} \subset \overline{\Omega}$.
Moreover, since the boundary $\partial \Omega$ is $C^1$, it is easy to see that,
by decreasing $\alpha$ and $r$ if necessary, we can assume that $\mathbf{C}_x \subset \overline{\Omega}$ and $\mathbf{C}'_x \subset \overline{\Omega^c}$, for all $x \in \partial \Omega$ with $|x-x_0| < 2r$. Then if $x \in \Omega$ is such that $|x-x_0| < r$ and $\mathbf{C}_x \not\subset \overline{\Omega}$ we can find a point $y \in \partial \Omega \cap \mathbf{C}_x$. Then the fact that $x \in \mathbf{C}'_y \cap \Omega$ and $|y - x_0| < 2r$ leads to contradiction.
\end{proof}

\begin{theorem}\label{thm: convergence for different shiftings}
Let $\Omega \subset \mathbb{R}^d$ be a domain with $C^1$ boundary, $(\epsilon_n)$ a sequence of positive real numbers converging to zero  and $(u_n)$  a sequence of continuous functions on $\overline{\Omega}$ satisfying $-\Delta_\infty^{\epsilon_n}u_n = f- c_f(\epsilon_n)$. Any limit $u$ of a subsequence of $(u_n)$ is a viscosity solution to
\begin{equation*}
\left\{\begin{aligned}
& -\Delta_\infty u = f-\overline{c}_f & \text{in} & \ \Omega,\\
&\nabla_\nu u = 0 & \text{on} & \ \partial \Omega.
\end{aligned}\right.
\end{equation*}
\end{theorem}

\begin{proof}
We denote the subsequence again by $(u_n)$.
To prove that $u$ is a solution to \eqref{eq:equation_continuous} we will check the assumptions from Definition \ref{def:viscosity_solutions} for local maxima. The conditions for local minima follow by replacing $u$ and $f$ by $-u$ and $-f$.
Let $\varphi \in C^\infty(\overline{\Omega})$ be a smooth function and $x_0 \in \overline{\Omega}$ be a point at which $u-\varphi$ has a strict local maximum. 
We will prove the claim for $x_0 \in \partial \Omega$. 
For the case $x_0 \in \Omega$ see either of the two proofs of Theorem 2.11 in \cite{AS} (in Sections 4 and 5).

Assume that $\nabla_{\nu} \varphi(x_0) > 0$. 
For $k$ large enough we can find  points $x_k$ so that $\lim_k x_k = x_0$ and such that $u_{k} -  \varphi$ has a local maximum at $x_k$. We can assume that for all $k$ we have $|\nabla \varphi(x_k)| > c$, for some $c>0$. 
Denote $\mathbf{v}_k=-\nabla\varphi(x_k)/|\nabla\varphi(x_k)|$ and for a given $k$  large enough, a point $\overline{x}_k \in \mathbf{B}(x_k,\epsilon_k)$ such that $\varphi(\overline{x}_k) = \min_{\mathbf{B}(x_k,\epsilon_k)}\varphi$. 
By Lemma \ref{lemma:comparison_discrete_to_infinite} (ii) and Lemma  \ref{lemma: cones} we can find $\alpha$ and $r$ such that for $k$ large enough we necessarily have $\overline{x}_k \in \mathbf{C}(x_k, \mathbf{v}_k, \alpha, r) \subset \overline{\Omega}$.
%
%
For such $k$ this readily implies that 
\begin{equation}\label{eq:big_and_small_laplacian}
-\widetilde{\Delta}_\infty^{\epsilon_k} \varphi(x_k) \leq -\Delta_\infty^{\epsilon_k} \varphi(x_k).
\end{equation}
Lemma \ref{lemma:comparison_discrete_to_infinite} further yields that there is a constant $C$ such that, for $k$ large enough
\begin{equation}\label{eq:laplace_close_to_normal_vector}
-\Delta_\infty \varphi(x_k) \leq -\widetilde{\Delta}_\infty^{\epsilon_k} \varphi(x_k) + C(1+c^{-1})\epsilon_k.
\end{equation}
Plugging \eqref{eq:big_and_small_laplacian} into \eqref{eq:laplace_close_to_normal_vector} we obtain
\begin{equation}\label{eq:laplace_close_to_normal_vector_1}
-\Delta_\infty \varphi(x_k) \leq -\Delta_\infty^{\epsilon_k} \varphi(x_k) + C(1+c^{-1})\epsilon_k.
\end{equation}
Since $u_k-\varphi$ has a  local maximum $x_k$ we have that 
\[
-\Delta_\infty^{\epsilon_k} \varphi(x_k) \leq -\Delta_\infty^{\epsilon_k} u_k(x_k) = f(x_k) - c_f(\epsilon_k).
\]
Inserting this into \eqref{eq:laplace_close_to_normal_vector_1} and taking the limit  as $k$ tends to infinity
implies that $-\Delta_\infty \varphi(x_0) \leq f(x_0) - \overline{c}_f$.
\end{proof}

\begin{proof}[Proof of Theorem \ref{thm:existence_of_continuous_solutions}]
The claim follows directly from Theorems \ref{thm:general_continuous_solutions} and \ref{thm: convergence for different shiftings}.
\end{proof}

To prove Theorem \ref{thm:uniqueness_of_shifting_constants} we will use Theorem 2.2 from \cite{ASS}. A general assumption in \cite{ASS} is that the boundary $\partial \Omega$ is decomposed into two disjoint parts, $\Gamma_D \neq \emptyset$ on which Dirichlet boundary conditions are given and $\Gamma_N$ on which vanishing Neumann boundary conditions are given. While the assumption $\Gamma_D \neq \emptyset$ is crucial for their existence result (Theorem 2.4 in \cite{ASS}) this assumption is not used in  Theorem 2.2 from \cite{ASS}. In the case when $\Gamma_D = \emptyset$ their result can be stated as follows.

\begin{theorem}[Case $\Gamma_D \neq \emptyset$ of Theorem 2.2 in \cite{ASS}]\label{thm:from_continuous_to_discrete}
Let $\Omega$ be a convex domain and $f, u\colon \Omega \to \mathbb{R}$ continuous functions such that $u$ is a viscosity subsolution to the equation \eqref{eq:equation_continuous}. Then for any $\epsilon > 0$ it holds that $-\Delta_\infty^\epsilon \overline{u}^\epsilon \leq \overline{f}^{2\epsilon}$ on $\overline{\Omega}$.
\end{theorem}

\begin{proof}[Proof of Theorem \ref{thm:uniqueness_of_shifting_constants}]
By Theorem \ref{thm:existence_of_continuous_solutions} the equation \eqref{eq:equation_continuous} has a solution when $\overline{c}_f=0$. 
Now assume that $u$ is a viscosity solution to \eqref{eq:equation_continuous}. 
By Theorem \ref{thm:from_continuous_to_discrete} we have that $-\Delta_\infty^\epsilon \overline{u}^\epsilon \leq f + \osc(f,2\epsilon)$. Now Lemma \ref{lemma:starting_from_the_subsolution} implies that $c_f(\epsilon) \geq -\osc(f,2\epsilon)$. Similarly one can obtain $c_f(\epsilon) \leq \osc(f,2\epsilon)$ and the claim follows by taking the limit as $\epsilon \downarrow 0$.
\end{proof}

\begin{remark}\label{rem:shifting_constant_as_linear_functional}
In the one dimensional case (say $\Omega=[0,r]$) the viscosity solutions to the equation \eqref{eq:equation_continuous} are standard solutions to the equation $-u''=f$ where $u'(0)=u'(r)=0$. It is clear that in this case $\overline{c}_f=\frac{1}{r}\int_0^rf(x)dx$. The fact that $\overline{c}_f$ is a linear functional of $f$ relies heavily on the fact that the infinity Laplacian in one dimension is a linear operator. For higher dimensional domains $f \mapsto \overline{c}_f$ is in general not a linear functional. Next we show that a two dimensional disc is an example of a domain on which $\overline c_f$ is a nonlinear functional of $f$ (essentially the same argument can be applied for balls in any dimension higher than one).

Take $\Omega$ to be a two dimensional disc of radius $r$ centered at the origin, and assume that on $C(\overline{\Omega})$ the mapping $f \mapsto \overline{c}_f$ is a linear functional. Since it is clearly a positive functional and  $\overline{c}_1=1$
by
 Riesz representation theorem it is of the form $\overline{c}_f = \int_{\overline{\Omega}} fd\mu$, for some probability measure $\mu$ on $\overline{\Omega}$.
Let $f$ be a radially symmetric function on $\overline{\Omega}$, that is  $f(x)=g(|x|)$, where $g \colon [0,r] \to \mathbb{R}$ is a continuous function. 
Let $u_n \colon \overline{\Omega} \to \mathbb{R}$ and $v_n \colon [0,r] \to \mathbb{R}$ be the sequences of game values with  the running payoff $f$ played on $\overline{\Omega}$ and the running payoff $g$ played on $[0,r]$ respectively, both games played with vanishing terminal payoff and  step size $\epsilon$.
Using induction and \eqref{eq:recursion} one can see that for any $n$ we have $u_n(x)=v_n(|x|)$. Thus, using the expression for $\overline{c}_f$ in the one dimensional case we have that $\mu$ is necessarily of the form $\mu(dx,dy)= \frac{dxdy}{r\pi\sqrt{x^2+y^2}}$, that is $\mu$ is a radially symmetric measure which assigns equal measure to any annulus of given width.

\begin{figure}
\includegraphics{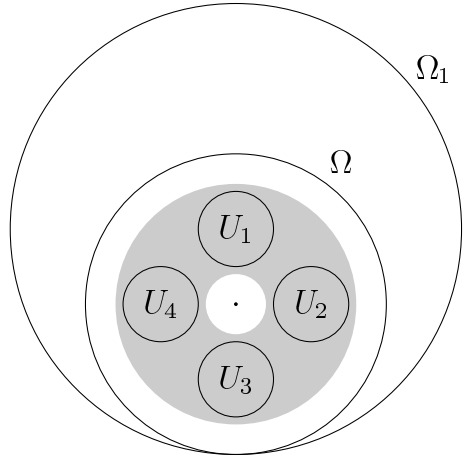}
\caption{Example of functions from Remark \ref{rem:shifting_constant_as_linear_functional}. Functions $f_1$, $f_2$, $f_3$ and $f_4$ have supports in discs $U_1$, $U_2$, $U_3$ and $U_4$ respectively. The shaded area is the support of the function $g$.}
\label{fig:no_functional}
\end{figure}

Next let $U_1$, $U_2$, $U_3$ and $U_4$ be disjoint discs with centers at $(0,r/2)$, $(r/2,0)$, $(0,-r/2)$, $(-r/2,0)$ and radii $r/4$. By $\Omega_1$ denote the smallest disc with the center $(0,r/2)$ which contains $\Omega$.
Let $f_1 \colon \overline{\Omega_1} \to [0,1]$ be a function with support in $U_1$ and which values are radially symmetric around $(0,r/2)$, decreasing with the distance from $(0,r/2)$ and such that $f_1(x)=1$ whenever the distance between $x$ and $(0,r/2)$ is no more than $r/4-\delta_1$.
Playing the game on $\overline{\Omega}_1$ with running payoff $f_1$ leads to
\begin{equation}\label{eq:shifting for the small disc}
\overline{c}_{f_1} \geq \frac{2}{3r}\int_0^{r/4-\delta_1}1dt = \frac{1}{6}-\frac{2\delta_1}{3r}.
\end{equation}
Let $w_n^1$ and $w_n$ be the sequences of game values with the running payoff  $f_1$ played on $\overline{\Omega}_1$ and the running payoff  $f_1|_{\overline{\Omega}}$ played on $\overline{\Omega}$ respectively, both games played with vanishing terminal payoff and  step size $\epsilon$.
It is clear that $w_n^1$ are radially symmetric functions on $\overline{\Omega}_1$ with values decreasing with the distance from $(0,r/2)$.  
Using this and the induction on $n$ one can see that $w_n^1(x) \leq w_n(x)$ for any $x \in \overline{\Omega}$. Therefore it holds that $\overline{c}_{f_1} \leq \overline{c}_{f_1|_{\overline{\Omega}}}$, which together with \eqref{eq:shifting for the small disc} implies
\[
\overline{c}_{f_1|_{\overline{\Omega}}} \geq \frac{1}{6}-\frac{2\delta_1}{3r}.
\]
Take $f_2$, $f_3$ and $f_4$ to be equal to the function $f_1$ rotated clockwise for $\pi/2$, $\pi$ and $3\pi/4$ respectively and $f=\sum_{i=1}^4f_i|_{\overline{\Omega}}$. By symmetry and  the assumed linearity, the function $f \colon \Omega \to \mathbb{R}$ satisfies
\begin{equation}\label{eq:shifting on the union}
\overline{c}_f \geq \frac{2}{3} - \frac{8\delta_1}{3r}.
\end{equation}
Now take $g \colon \overline{\Omega} \to [0,1]$ to be a radially symmetric function on $\overline{\Omega}$, such that $g(x) > 0$ if and only if $r/4-\delta_2 \leq |x| \leq 3r/4+\delta_2$ and such that $f \leq g$. Again by the assumption we have
\begin{equation}\label{eq:shifting on the strip}
\overline{c}_g \leq \frac{1}{r}\int_{r/4-\delta_2}^{3r/4+\delta_2} 1dt = \frac{1}{2} + \frac{2\delta_2}{r}.
\end{equation}
Since $\overline{c}_f \leq \overline{c}_g$, for small $\delta_1$ and $\delta_2$, inequalities \eqref{eq:shifting on the union} and \eqref{eq:shifting on the strip} give a contradiction with the assumed linearity of the mapping $f \mapsto \overline{c}_f$.
See Figure \ref{fig:no_functional}.
\end{remark}

\section{Uniqueness discussion}\label{section: uniqueness discussion}

As Figure \ref{bowties} illustrates, in the graph case, once we are given $f$, we do not
always have a unique solution to 
\[
u(x) - \frac{1}{2}(\min_{y \sim x}u(y) + \max_{y \sim x}u(y)) = f(x),
\] 
even in the case of a
finite graph with self loops.   When the running payoff at each vertex
is $f$, the corresponding ``optimal play'' will make $u(x_k)$ plus
the cumulative running payoff a martingale ($x_k$ is the position of the token at the $k$th step).  Under such play, the
players may spend all of their time going back and forth between the
two vertices in one of the black-white pairs in Figure \ref{bowties}.
(In the case of the second function shown, the optimal move choices
are not unique, and the players may move from one black-white pair to
another.)  The basic idea is that as the players are competing within
one black-white pair, neither player has a strong incentive to try to
move the game play to another black-white pair.

A continuum analog of this construction appeared in Section 5.3 of
[PSSW], where it was used to show non-uniqueness of solutions to
$\Delta_\infty u = g$ with zero boundary conditions.  (In this case,
$g$ was zero except for one ``positive bump'' and one symmetric
``negative bump''.  Game players would tug back and forth between the
two bumps, but neither had the ability to make the the game end
without giving up significant value.)  It is possible that this
example could be adapted to give an analogous counterexample to
uniqueness in our setting (i.e., one could have two opposite-sign
pairs of bumps separated by some distance on a larger domain with free
boundary conditions --- and players pulling back and forth within one
pair of bumps would never have sufficient incentive to move to the
other pair of bumps --- and thus distinct functions such as those
shown in Figure \ref{bowties} could be obtained).  However, the
construction seems rather technical and we will not attempt it here.

\begin {figure}[!ht]
\begin {center}
\includegraphics [width=4in]{bowties.epsi}
\caption {\label{bowties} The difference equation \eqref{eq:equation_discrete} does not
always have a unique (up to additive constant) solution.  On each of
the three copies of the same graph shown above (assume self-loops are
included, though not drawn), we
define $f$ to be the function which is $-1$ on black vertices, $1$ on
white vertices, and $0$ on gray vertices.  Then each of the three
functions illustrated by numbers above solves \eqref{eq:equation_discrete}.  In each
case, the value at a gray vertex is the average of the largest and
smallest neighboring values. The value at a black (white) vertex is
$1$ less (more) than the average of the largest and smallest
neighboring values.
}
\end {center}
\end {figure}

\section*{Acknowledgments}
The authors would like to thank Charles Smart and Scott Armstrong for helpful discussions and references. The first and fourth author would like to thank Microsoft Research where part of this work was done. The fourth author was supported by NSF grant DMS-0636586.


\end{document}